\documentclass[12pt]{amsart}
\usepackage{amsmath, amsthm, amsfonts, amssymb, graphicx, hyperref, bm,verbatim,mathrsfs,siunitx,natbib,tikz-cd,mathabx,mathtools,dsfont}
\usepackage{stmaryrd,enumerate}
\theoremstyle{plain}
\usepackage{geometry}
\newtheorem{thm}{Theorem}[section]
\newtheorem{lem}[thm]{Lemma}
\newtheorem{prop}[thm]{Proposition}
\newtheorem{cor}[thm]{Corollary}

\newtheorem{conj}{Conjecture}
\newtheorem*{que}{Question}

\numberwithin{sublem}{thm} 
\numberwithin{equation}{section}
\renewcommand{\S}{S(\mathcal{A}, \mathcal{P})}

\renewcommand*{\d}{\, \textup{d}}
\renewcommand{\gcd}{\operatorname{gcd}}

\newcommand{\e}{\varepsilon}

\numberwithin{equation}{section}
\makeatletter
\@namedef{subjclassname@2020}{%
  \textup{2020} Mathematics Subject Classification}
\makeatother
\usepackage[T1]{fontenc}
\numberwithin{equation}{section}
\let\underbrace\LaTeXunderbrace

\begin{document}
\title{On the Number of Prime Factors of Consecutive Integers}
\author{Cheuk Fung (Joshua) Lau}
\address{Mathematical Institute, University of Oxford, Radcliffe Observatory Quarter, Woodstock Rd, Oxford OX2 6GG, UK}
\email{joshua.cf.lau@gmail.com}
\begin{abstract}
We prove that there are infinitely many $n$ such that $\omega(n+k) \ll \log k$ for all integers $k \ge 2$. This improves on a result of Tao and Ter\"{a}v\"{a}inen \cite{tao_quantitative_2025}, who have $O(k)$ in place of $O(\log k)$. As corollaries, we make progress on a number of questions posed by Erd\H{o}s. The proof is based on a quantitative refinement of the Tao-Ter\"{a}v\"{a}inen probabilistic argument, combining a refined sieve construction with upper bounds for high moments and the resulting tail estimates. Moreover, we formulate a conjecture on integers with many prime factors based on Cram\'{e}r-type random models. Assuming this conjecture, the main bound is essentially sharp.
\end{abstract}
\maketitle
\tableofcontents

\section{Introduction}
In this paper we are interested in when there are strings of consecutive integers each having `few' prime factors. More specifically, we study variants of the following question.

\begin{que}
Given a function $f:\mathbb{N} \to \mathbb{N}$, are there infinitely many integers $n$ such that $\omega(n+k)\le f(k)$ for all integers $k \ge 1$?
\end{que}
Here $\omega(n)=\sum_{p\mid n}1$ is the number of distinct prime factors of $n$, and one can ask similar questions for the closely related functions $\Omega(n):=\sum_{p^j\mid n}1$ (the number of prime factors counting multiplicity) or $\tau(n):=\sum_{d\mid n}1$ (the number of divisors of $n$). One can also ask for $\omega(n-k)\le f(k)$ for all $1\le k<n$. The methods used here also apply to this formulation, since reversing the signs of the shifts just reflects the relevant congruence classes.

Erd\H{o}s asked several questions of this type, including:

\begin{conj}[{\cite[Problem 4.33.2.II.(1)]{erdos1974remarks}}, {\cite[Theorem 1.1]{tao_quantitative_2025}}] \label{conj:erdos248}
    There are infinitely many $n$ such that $$\omega(n+k) \le \Omega(n+k) \ll k$$ for all integers $k\ge 1$.
\end{conj}

\begin{conj}[{\cite[p. 71]{erdHos1979some}}] \label{conj:erdos413}
    There are infinitely many $n$ such that $$\omega(n-k) \le k$$ for all integers $1 \le k <n$.
\end{conj}

\begin{conj}[{\cite[Problem 4.33.2.II.(2)]{erdos1974remarks}}] \label{conj:erdos826}
    There are infinitely many $n$ such that $\tau(n+k) \ll k$ for all integers $k \ge 1$.
\end{conj}

\begin{conj}[{\cite[(3)]{erdHos1979some}}] \label{conj:erdos679}
   For every $\varepsilon>0$, there exists an integer
$k_0=k_0(\varepsilon)\ge 3$ such that there are infinitely
many $n$ satisfying both
\[
    \omega(n-k)\le (1+\varepsilon)\frac{\log k}{\log\log k}, \qquad \Omega(n-k)\le (1+\varepsilon)\frac{\log k}{\log 2}
\]
for all integers $k_0\le k<n$.
\end{conj}

These are listed as Erd\H{o}s Problem \#248, \#413, \#826 and \#679 respectively on the Erd\H{o}s Problems website\footnote{See Thomas Bloom's \emph{Erd\H{o}s Problems},
\url{https://www.erdosproblems.com/}, Problems~248, 413, 826,
and~679, corresponding to Conjectures \ref{conj:erdos248}, \ref{conj:erdos413}, \ref{conj:erdos826}, and \ref{conj:erdos679}, respectively.}, maintained by Thomas Bloom, which catalogues problems posed by Erd\H{o}s. Conjecture 1 was recently solved by \cite[Theorem 1.1]{tao_quantitative_2025}.

In this paper, we prove the following result, which is an improvement on Theorem 1.1 of \cite{tao_quantitative_2025}.
\begin{thm} \label{thm:mainthm}
  There exists a positive constant $C$ such that, for infinitely many positive integers $n$, one has $\omega(n+k) \le \Omega(n+k) \le C\log k$ for every integer $k \ge 2$.
\end{thm}
Using random models, we conjecture this is best possible up to a constant.
\begin{conj} \label{conj:mainthm}
    For every $\varepsilon>0$ and every sufficiently large integer $n$
in terms of $\varepsilon$, there exists an integer $k\geq2$ such that
\[
 \omega(n+k)>(1-\varepsilon)\log k.
\]
\end{conj}

As a corollary of Theorem \ref{thm:mainthm}, we have a weaker version of Conjecture \ref{conj:erdos826} (Erd\H{o}s Problem \#826).
\begin{cor} \label{cor:826divisor}
    There is an absolute constant $C$ such that there are infinitely many $n$ satisfying $\tau(n+k) \ll k^{C}$ for all integers $k \ge 2$. 
\end{cor}
The same random model predicts that the exponent $1$
in Conjecture \ref{conj:erdos826} is optimal. In Section \ref{sec:conditionalfalsity} we give a
short-interval conjectural assertion which implies that
no bound $\tau(n+k)\ll k^{1-\varepsilon}$, with fixed
$\varepsilon>0$, can hold for all $k\ge 2$ for infinitely
many $n$.

With the same proof, a version of Theorem \ref{thm:mainthm} for $n-k$ also holds.
\begin{thm} \label{thm:mainthm_minus}
  There exists a positive constant $C$ such that, for infinitely many positive integers $n$, one has $\omega(n-k) \le \Omega(n-k) \le C\log k$ for every integer $1< k <n$.
\end{thm}
By translating $n$, Theorems \ref{thm:mainthm} and \ref{thm:mainthm_minus} also hold with
$C\log(k+2)$ in place of $C\log k$, for all integers
$k\ge 0$ and $0\le k<n$, respectively.

We also speculate that the analogous version of Conjecture \ref{conj:mainthm} holds.
\begin{conj} \label{conj:newerdos679}
   For every $\varepsilon>0$ and every sufficiently large integer $n$
in terms of $\varepsilon$, there exists an integer $1<k<n$ such that
\[
 \omega(n-k)>(1-\varepsilon)\log k.
\]
\end{conj}
The corresponding assertions for $\Omega$ are elementary.
\begin{prop}
For every $\varepsilon>0$ and every sufficiently large integer $n$
in terms of $\varepsilon$, there are integers $k_+\geq2$ and
$1<k_-<n$ such that
\[
 \Omega(n+k_+)>(1-\varepsilon)\frac{\log k_+}{\log2},
 \qquad
 \Omega(n-k_-)>(1-\varepsilon)\frac{\log k_-}{\log2}.
\]
\end{prop}
\begin{proof}
Choose $2^a\in[2n,4n]$ and $2^b\in[n/4,n/2]$,
and set $k_+:=2^a-n$ and $k_-:=n-2^b$.
Then $n\leq k_+\leq3n$ and $n/2\leq k_-\leq3n/4$, while
\[
 \Omega(n+k_+)=a=\frac{\log n}{\log2}+O(1),\qquad
 \Omega(n-k_-)=b=\frac{\log n}{\log2}+O(1).
\]
Since $\log k_\pm=\log n+O(1)$, both assertions follow.
\end{proof}
The bound for $\omega(n-k)$ in Theorem \ref{thm:mainthm_minus} is weaker
than that of Conjecture \ref{conj:erdos679} by a factor
of order $\log\log k$. The bound for $\Omega(n-k)$
has the same order of magnitude as the second assertion,
but does not attain the conjectured constant $1/\log 2$.
In Section \ref{sec:conditionalfalsity} we formulate a short-interval conjecture
which, if true, would imply Conjecture \ref{conj:erdos679} is false.

Using Theorem \ref{thm:mainthm_minus}, we immediately obtain a corollary, which is a weaker version of Conjecture \ref{conj:erdos413} (Erd\H{o}s Problem \#413).
\begin{cor}
    There exists an absolute integer $k_0\ge 2$ such that
there are infinitely many $n$ satisfying
\[
    \omega(n-k)\le k
\]
for all integers $k_0\le k<n$.
\end{cor}

\section{Outline}
In this section, we outline the proof of Theorem \ref{thm:mainthm} for $\omega$. The strategy for $\Omega$ is similar. Following the broad strategy of \cite[Section 2]{tao_quantitative_2025}, we construct a random variable $\mathbf{n}$ taking values in $[x, 2x]$ and show that with positive probability, $\omega(\mathbf{n}+k) \ll \log k$ for all $k \ge 2$. By the union bound, it suffices to prove that
\begin{equation} \label{eqn:unionbound}
\sum_{k=2}^\infty \mathbb{P} \left( \omega(\mathbf{n}+k) \ge C \log k \right) < 1
\end{equation}
for a sufficiently large constant $C$. Similar to \cite[Page 11]{tao_quantitative_2025}, the terms where $k$ is very large (e.g., $k \ge x^{1/100}$) are handled by trivial bounds
$$\omega(\mathbf{n}+k) \le \log(\mathbf{n}+k) \le \log(2x+k) \ll \log k.$$
The main challenge lies in the range $k < x^{1/100}$.
To prove \eqref{eqn:unionbound}, we construct $\mathbf{n}$ by weighting the uniform distribution on $[x, 2x] \cap \mathbb{Z}$ with a product of Selberg-type sieve weights. Specifically, we choose $n \in [x, 2x]$ with probability proportional to
\[
w(n) := \mathds{1}_{p^4 \mid n \, \forall p \le w} \prod_{k=1}^K w_{R_k}(n+k)^2,
\]
where $w_{R_k}(n+k)$ are sieve weights of Goldston–-Pintz–-Yıldırım type, $K = \lfloor (\log x)^{1/1000}\rfloor$, and $w=0.15 \log x$. Morally, these weights act as a proxy for the indicator function that $n+k$ has no prime factors between $w$ and the sieve level $R_k$. We discuss the choice of $R_k$ later.

The growing dimension $K$ does not lead to the usual exponential losses from a high-dimensional sieve. Indeed, all primes $p\leq w$ have already been treated separately, while for $p>w>K$ the $K$ residue classes being sieved are distinct, and the local factor $1-K/p$ differs from $(1-1/p)^K$ by only $1+O(K^2/p^2)$. Since $K^2=o(w)$, the product of these errors over $p>w$ is $\exp(O(K^2/w))=1+o(1)$, so the sieve behaves essentially as a product of one-dimensional sieves.

A key technical requirement for this construction is that the product of the sieve levels must satisfy $\prod_{k=1}^K R_k < x^\theta$ for some $\theta < 1$. However, to ensure that $n+k$ has few prime factors, we want each $R_k$ to be as large as possible. To balance these requirements, we choose a polynomial decay for the sieve levels
\[
 R_k=x^{c/k^{50}}\quad(1\leq k\leq K),
\]
while $R_k=w$ for $K<k\leq x^{1/100}$, where $c$ is a small constant. This choice of polynomial decay in $k$ in the exponent is a quantitative refinement of \cite[(2.2)]{tao_quantitative_2025} and is essential for the sharp bounds required in Theorem \ref{thm:mainthm}.

Under this weighting, an integer $\mathbf n+k$ behaves like a random integer conditioned to have no prime factors between $w$ and $R_k$. If $\mathbf n+k$ has no prime factors below $R_k$, it trivially has at most $\log(3x)/\log R_k \ll k^{50}$ prime factors, but this is far too large for our purposes. Instead, the independent-divisibility model suggests that a typical integer with no prime factors below $R_k$ has
\[
\log \left( \frac{\log x}{\log R_k} \right) = \log(k^{50}/c) \asymp \log k \qquad (2 \le k \le K)
\]
prime factors. Here the
heuristic is that divisibility by a prime $p>R_k$ has
probability approximately $1/p$. To find $n$ such that $n+k$ simultaneously exhibit this typical behaviour for $2 \le k \le x^{1/100}$, we establish upper bounds
for moments of order proportional to $\log k$ under the sieve
measure. Applying Chebyshev's inequality to these moment bounds,
with sufficiently large fixed constants, gives
\[
 \mathbb P\bigl(\omega(\mathbf n+k)\ge C\log k\bigr)
 <\frac{6}{\pi^2k^2}
 \qquad(2\leq k\leq x^{1/100}).
\]
The sum of these upper bounds over $k\geq2$ is less than one,
and larger shifts are handled by elementary deterministic bounds.
This proves that the desired inequalities hold simultaneously
with positive probability. The use of higher moments supplies
the stronger tail bounds needed for the logarithmic dependence
on $k$. By contrast, \cite[Section 2.1]{tao_quantitative_2025} obtained a comparable bound for $\mathbb{P}(\omega(\mathbf{n}+k) \ge Ck)$ using bounded-order moment estimates.

\section{Acknowledgements}
I am very grateful to my supervisor, James Maynard, for suggesting this problem and for many helpful comments and discussions. This work was supported by the Oxford–Croucher Scholarship.

\section{Notation}
Let $x$ be an asymptotic parameter tending to $\infty$. We write $X\ll Y$, $Y\gg X$, or $X=O(Y)$ to mean that $|X|\le CY$ for some constant $C$, and $X=o(Y)$ to mean that $|X|\le c(x)Y$ for some $c(x)\to 0$ as $x\to\infty$. We write $X\asymp Y$ for $X\ll Y\ll X$. For $j \in \mathbb{N}$, we use $\log_j x$ to denote $\underbrace{\log \cdots \log}_{j \text{ times}} \, x$. For a prime $p$ and a positive integer $n$, we use $\nu_p(n)$ to denote the $p$-adic valuation of $n$, i.e. $p^{\nu_p(n)} \mid n$ but $p^{\nu_p(n)+1} \nmid n$. For a natural number $n$, we use $[n]$ to denote the set $\{1,2,\ldots,n\}$. If $m_1,...,m_k$ are natural numbers, we use $[m_1,...,m_k]$ to denote the least common multiple $\operatorname{lcm}(m_1,...,m_k)$, and we use $(m_1,\dots,m_k)$ to denote the greatest common divisor $\gcd(m_1,\ldots,m_k)$. For $a,b>0$, we define $\omega_{a<\cdot \le b}(n)=\sum_{a<p \le b} \mathds{1}_{p \mid n}$ if $a<b$, and 0 otherwise.

\section{Reducing Theorem \ref{thm:mainthm} to Proposition \ref{prop:mainpropaxioms}}
Using the union bound, we show that it suffices to prove the following proposition.
\begin{prop} \label{prop:mainlargedeviations}
    There exists a constant $C>0$ such that for $x$ sufficiently large, we may construct a random variable $\mathbf{n}$ taking values in $[x,2x] \cap \mathbb{Z}$ such that for every integer $2 \le k \le x^{1/100}$,
\begin{align}
  \mathbb{P}(\Omega(\mathbf{n}+k)>C \log k) < \frac{6}{\pi^2k^2}. \label{eqn:mainlargedeviations}
\end{align}
\end{prop}
\begin{proof}[Proof of Theorem \ref{thm:mainthm} using Proposition \ref{prop:mainlargedeviations}]
From \eqref{eqn:mainlargedeviations} and the union bound, for sufficiently large $C_0$ and $x \in \mathbb{R}^+$ large in terms of $C_0$, we have
  \[
  \mathbb{P} \left(\bigcup_{2 \le k \le x^{1/100}} \left\{ \Omega(\mathbf{n}+k)>C_0 \log k \right\}\right) \le \sum_{2 \le k \le x^{1/100}} \mathbb{P}(\Omega(\mathbf{n}+k) >C_0 \log k) <1,
  \]
  and so
  \[
  \mathbb{P} \left(\bigcap_{2 \le k \le x^{1/100}} \left\{ \Omega(\mathbf{n}+k) \le C_0 \log k \right\}\right) >0.
  \]
  Therefore, we can find $n \in [x,2x]$ such that $\Omega(n+k) \le C_0 \log k$ for all $2 \le k \le x^{1/100}$. For these choices of $n$, we also have for $x^{1/100}<k \le 2x$,
  \[
  \Omega(n+k) \le \frac{\log (n+k)}{\log 2} \le \frac{\log(4x)}{\log 2} \ll \log x \ll \log k,
  \]
  and for $k>2x$,
  \[
  \Omega(n+k) \le \frac{\log(n+k)}{\log 2} \le \frac{\log(2k)}{\log 2} \ll \log k.
  \]
  Therefore, choosing $C$ sufficiently large in terms of $C_0$, we are done.
\end{proof}
Next, we will show that it suffices to split into four contributions.
\begin{prop} \label{prop:largedeviations}
For every $C_2 \ge 16$, there exist $C_1 > 0$ and $A > 0$ such that the following holds. For every x sufficiently large in terms of $C_2$, $C_1$, and $A$, and $2 \le k \le x^{1/100}$, define the parameters
\[
w := 0.15 \log x, \qquad 
T_k := x^{1/{10A \log k}},
\]
and
\[
R_k :=
\begin{cases}
x^{1/(100(A+1) k^{50})}, & \text{if } k \le \lfloor (\log x)^{1/1000}\rfloor,\\[6pt]
w, & \text{if } \lfloor (\log x)^{1/1000}\rfloor < k \le x^{1/100}.
\end{cases}
\]
For these choices of $A,C_1$, and each sufficiently large $x$, one can construct a random variable
$\mathbf n$ taking values in $[x,2x] \cap \mathbb{Z}$ such that the following
properties hold simultaneously for all $2 \le k \le x^{1/100}$. With probability 1, the integer $\mathbf{n}$ is divisible by $p^4$ for every prime $p \le w$. Moreover, the following additional properties hold:
\begin{align}
  \mathbb{P}\left(\omega_{w<\cdot \le R_k}(\mathbf{n}+k) \ge C_1 \log k\right) &\ll \frac{1}{C_2 k^2},\label{eqn:medium_largedeviations}\\
\mathbb{P} \left( \sum_{R_k<p \le T_k} \left( \mathds{1}_{p \mid \mathbf{n}+k}-\frac1p \right) \ge C_1 \log k \right) &\ll \frac{1}{C_2 k^2},\label{eqn:large_largedeviations}\\
\mathbb{P} \left(\sum_{j \ge 2} \sum_{w<p \le T_k} \mathds{1}_{p^j \mid \mathbf{n}+k} \ge C_1 \log k\right)  &\ll \frac{1}{C_2k^2},
\label{eqn:power_largedeviations_medium}\\
\mathbb{P}\Bigg(\sum_{\substack{p \le w\\p^4 \mid k}} (\nu_p(\mathbf{n}+k)-\nu_p(k)) \ge C_1 \log k \Bigg) &\le \frac{3}{2\pi^2 k^2}. \label{eqn:power_largedeviations_tiny}
\end{align}
Here the implied constants are absolute constants and do not depend on $A$, $C_1$, or $C_2$.
\end{prop}
\begin{proof}[Proof of Proposition \ref{prop:mainlargedeviations} using Proposition \ref{prop:largedeviations}]
    Fix $C_2 \ge 16$ and choose $C_1$ and $A$ as in Proposition \ref{prop:largedeviations}. Since for sufficiently large $x$, 
    \[
    \sum_{p>T_k} \nu_p(\mathbf{n}+k) \le \frac{\log(\mathbf{n}+k)}{\log T_k} \ll_A \log k,
    \]
    we have
\begin{align*}
  \Omega(\mathbf{n}+k) & \le \sum_{\substack{p \le w\\p^4 \nmid k}} \nu_p(\mathbf{n}+k)+\sum_{w<p \le R_k} \mathds{1}_{p \mid \mathbf{n}+k} + \sum_{R_k <p \le T_k} \mathds{1}_{p \mid \mathbf{n}+k}+O_A(\log k)\\
  &\qquad+\sum_{j \ge 2} \sum_{w<p \le T_k} \mathds{1}_{p^j \mid \mathbf{n}+k}+\sum_{\substack{p \le w\\p^4 \mid k}} \nu_p(\mathbf{n}+k)
\end{align*}
We first treat the first and last terms. Under the condition that $p^4 \mid \mathbf{n}$ for every prime $p \le w$ (which occurs with probability 1), we have 
\[
\sum_{\substack{p \le w\\p^4 \nmid k}} \nu_p(\mathbf{n}+k) \le \sum_{\substack{p \le w\\p^4 \nmid k}} \nu_p(k),
\]
which implies
\begin{align*}
    \sum_{\substack{p \le w\\p^4 \nmid k}} \nu_p(\mathbf{n}+k)+\sum_{\substack{p \le w\\p^4 \mid k}} \nu_p(\mathbf{n}+k)&\le \sum_{p \le w} \nu_p(k)+\sum_{\substack{p \le w\\ p^4 \mid k}} (\nu_p(\mathbf{n}+k)-\nu_p(k)) \\
    &\le \Omega(k) +\sum_{\substack{p \le w\\ p^4 \mid k}} (\nu_p(\mathbf{n}+k)-\nu_p(k))\\
    &\le 2\log k+\sum_{\substack{p \le w\\ p^4 \mid k}} (\nu_p(\mathbf{n}+k)-\nu_p(k)).
\end{align*}
Altogether, this gives
\begin{align*}
  \Omega(\mathbf{n}+k) \le &\sum_{w<p \le R_k} \mathds{1}_{p \mid \mathbf{n}+k} + \sum_{R_k <p \le T_k} \mathds{1}_{p \mid \mathbf{n}+k}+O_A(\log k)\\
  &\qquad +\sum_{j \ge 2} \sum_{\substack{w<p \le T_k}} \mathds{1}_{p^j \mid \mathbf{n}+k}+\sum_{\substack{p \le w\\p^4 \mid k}} (\nu_p(\mathbf{n}+k)-\nu_p(k)),
\end{align*}
under the condition that $p^4 \mid \mathbf{n}$ for every prime $p \le w$.
For the second term, we write
\[
\omega_{R_k<\cdot \le T_k}(\mathbf{n}+k) = \sum_{R_k<p \le T_k} \left( \mathds{1}_{p \mid \mathbf{n}+k}-\frac1p \right)+\sum_{R_k<p \le T_k} \frac1p.
\]
By Mertens' Theorem, we have
\[
\sum_{R_k<p \le T_k} \frac1p \le \max\left\{\log \frac{\log T_k}{\log R_k},0 \right\}+O \left( \frac{1}{\log R_k} \right),
\]
and for $2 \le k \le \lfloor (\log x)^{1/1000}\rfloor$ we have
\[\log \frac{\log T_k}{\log R_k} \le \log \left( \frac{\frac{1}{10A \log k} \log x}{\frac{1}{100(A+1)k^{50}} \log x} \right)\ll\log k,\]
while for $\lfloor (\log x)^{1/1000}\rfloor<k \le x^{1/100}$ we have
\[
\max\left\{\log \frac{\log T_k}{\log R_k},0\right\} \le \log \log x \le 1000 \log (\log x)^{1/1000} \ll \log k.  
\]
Therefore, we have
\begin{align}
  \Omega(\mathbf{n}+k) \le &\sum_{w<p \le R_k}  \mathds{1}_{p \mid \mathbf{n}+k} + \sum_{R_k <p \le T_k} \left( \mathds{1}_{p \mid \mathbf{n}+k}-\frac1p \right)+O_A(\log k) \label{eqn:splittingOmega}\\
  &\qquad +\sum_{j \ge 2} \sum_{\substack{w<p \le T_k}} \mathds{1}_{p^j \mid \mathbf{n}+k}+\sum_{\substack{p \le w\\p^4 \mid k}} (\nu_p(\mathbf{n}+k)-\nu_p(k)), \notag{}
\end{align}
under the condition that $p^4 \mid \mathbf{n}$ for every prime $p \le w$. 
Choose $C$ sufficiently large in terms of $A$ and $C_1$. Note that $\Omega(\mathbf{n}+k)>C \log k$ implies
\[
\sum_{w<p \le R_k}  \mathds{1}_{p \mid \mathbf{n}+k} + \sum_{R_k <p \le T_k} \left( \mathds{1}_{p \mid \mathbf{n}+k}-\frac1p \right)
 +\sum_{j \ge 2} \sum_{\substack{w<p \le T_k}} \mathds{1}_{p^j \mid \mathbf{n}+k}+\sum_{\substack{p \le w\\p^4 \mid k}} (\nu_p(\mathbf{n}+k)-\nu_p(k))>4C_1 \log k,
\]
so by the pigeonhole principle, at least one of the sums above is greater than $C_1 \log k$. Using Proposition \ref{prop:largedeviations} and the union bound, by choosing $C_2$ sufficiently large we get
    \begin{align*}
        \mathbb{P}(\Omega(\mathbf{n}+k)> C \log k)\le \frac{3}{2 \pi^2k^2}+O \left( \frac{1}{C_2k^2} \right)<\frac{6}{\pi^2k^2}
    \end{align*}
    for all integers $2 \le k \le x^{1/100}$.
\end{proof}
Using Chebyshev's inequality, that is, for any integer $s \ge 3$, a random variable $X$ with $\mathbb{E}|X|,\mathbb{E}|X-\mathbb{E}X|^s<\infty$, and $r>0$ real,
\[
\mathbb{P}(|X-\mathbb{E}X| \ge r) \le \frac{\mathbb{E}|X-\mathbb{E}X|^s}{r^s},
\]
we will reduce Proposition \ref{prop:largedeviations} to upper bounds on high moments.
\begin{lem} \label{lem:moment_medium_large}
There exist absolute constants $A_0,D_1,D_2,D_3\ge1$ with the following property. Suppose $A\ge A_0$, $2 \le k \le x^{1/100}$ and $s_1,s_2,s_3$ are integers satisfying
$\max\{3,\log k\}\leq s_1,s_2,s_3\leq A\log k$. Then for any $x \in \mathbb{R}^+$ sufficiently large in terms of $A$, define the parameters
\[
w := 0.15 \log x, \qquad 
T_k := x^{1/{10A \log k}},
\]
and
\[
R_k :=
\begin{cases}
x^{1/(100(A+1) k^{50})}, & \text{if } k \le \lfloor (\log x)^{1/1000}\rfloor,\\[6pt]
w, & \text{if } \lfloor (\log x)^{1/1000}\rfloor < k \le x^{1/100}.
\end{cases}
\]
For each fixed $A$, one can construct a random variable
$\mathbf n$ taking values in $[x,2x] \cap \mathbb{Z}$ such that the following
properties hold simultaneously for all admissible $k,s_1,s_2,s_3$.
With probability $1$, the integer $\mathbf n$ is divisible by $p^4$
for every prime $p\leq w$. Moreover,
\begin{align}
  \mathbb{E} \left| \sum_{w<p \le R_k} \mathds{1}_{p \mid \mathbf{n}+k} \right|^{s_1} &\ll (D_1s_1)^{s_1},\label{eqn:moment_medium_primes}\\
  \mathbb{E} \left| \sum_{R_k<p \le T_k} \left( \mathds{1}_{p \mid \mathbf{n}+k} -\frac1p \right) \right|^{2s_2} &\ll  \left( D_2A \log k \right)^{2s_2},\label{eqn:moment_large_primes}\\
  \mathbb{E} \left|\sum_{j \ge 2} \sum_{w<p \le T_k} \mathds{1}_{p^j \mid \mathbf{n}+k}\right|^{s_3} &\ll \left(D_3s_3 \right)^{s_3}, \label{eqn:moment_power_primes_medium}\\
    \mathbb{P}(\nu_p(\mathbf{n}+k)-\nu_p(k) \geq h_p \ \forall p \in \mathcal{P})&\le \prod_{p \in \mathcal{P}} p^{-h_p}+x^{-0.3},
    \label{eqn:mainpropE}
\end{align}
for any $\mathcal{P} \subseteq \{p \le w:p^4 \mid k\}$, and $h_p \in \mathbb{Z}^+$ for $p \in \mathcal{P}$. Here the implied constants are absolute and independent of $A$.
\end{lem}
In the calculation of $s$-th moments it will be helpful to have the following estimate.
\begin{lem}[Stirling numbers of the second kind] \label{lem:stirling_second_kind}
  For integers $s\ge 1$ and $1\le t\le s$, the Stirling number of the second kind $\displaystyle \left\{ {s \atop t} \right\}$ is the number of ways to partition a set of $s$ labelled objects into $t$ nonempty unlabelled boxes.
  It satisfies the bound
  \begin{align}\label{eqn:stirlingnumber_firsteq}
  \left\{ {s \atop t} \right\} \ll \left(\frac{s}{\log s}\right)^s
  \end{align}
  for all sufficiently large $s$ and $1\le t \le s$, and the 
  identity for all $M,z \ge 0$,
  \begin{align}\label{eqn:stirlingnumber_thirdeq}
  \sum_{q \ge 0} \left( \sum_{j=0}^q \left\{q \atop j \right\} M^j \right)\frac{z^q}{q!} =\exp\left( M(e^z-1) \right).
  \end{align}
\end{lem}
\begin{proof}
  See \cite[(7.54)]{graham1994concrete} for \eqref{eqn:stirlingnumber_thirdeq}. For \eqref{eqn:stirlingnumber_firsteq}, put $z=\frac12\log s$, then \eqref{eqn:stirlingnumber_thirdeq} with $M=1$ gives
  \[
    \left\{\begin{matrix}s\\t\end{matrix}\right\}
    \le \sum_{j=0}^{s}
        \left\{\begin{matrix}s\\j\end{matrix}\right\}
    \le s!z^{-s}\exp(e^z-1)
    \ll \left(\frac{s}{\log s}\right)^s,
\]
where the last estimate follows from Stirling's formula.
\end{proof}
\begin{proof}[Proof of Proposition \ref{prop:largedeviations} using Lemma \ref{lem:moment_medium_large}]
Fix $C_2\geq 16$, and let $A_0$ be the absolute constant in Lemma \ref{lem:moment_medium_large}. Put
\[
 B_2:=2+\frac{\log C_2+1}{\log 2},
 \qquad A:=\max\{A_0,B_2\}.
\]
Let $D_1,D_2,D_3$ be the absolute constants in Lemma \ref{lem:moment_medium_large}, and choose
\[
 C_1\geq
 \max\left\{46,
 e\max\{D_1A,D_2A,D_3A\}\right\}.
\]
For each $k\geq 2$, set
\[
 s_1=s_2=s_3=s:=\lceil \log(C_2k^2)\rceil.
\]
Then $s\geq 3$, $s\geq\log k$, and
\[
 s\leq \log(C_2k^2)+1\leq B_2\log k\leq A\log k.
\]
We now use Lemma \ref{lem:moment_medium_large}. By \eqref{eqn:moment_medium_primes} and Chebyshev's inequality,
\[
 \mathbb P\bigl(\omega_{w<\cdot\leq R_k}(\mathbf n+k)
                  \geq C_1\log k\bigr)
 \ll \left(\frac{D_1s}{C_1\log k}\right)^s
 \leq e^{-s}\leq (C_2k^2)^{-1}.
\]
Similarly, \eqref{eqn:moment_large_primes} gives
\[
 \mathbb P\left(
   \sum_{R_k<p\leq T_k}\left(\mathds{1}_{p\mid\mathbf n+k}-\frac1p\right)
       \geq C_1\log k\right)
 \ll \left(\frac{D_2A}{C_1}\right)^{2s}
 \leq e^{-2s}\leq (C_2k^2)^{-1}.
\]
Using \eqref{eqn:moment_power_primes_medium}, we also have
\begin{align*}
 \mathbb P\left(\sum_{j\geq2}\sum_{w<p\leq T_k}
                        \mathds{1}_{p^j\mid\mathbf n+k}\geq C_1\log k\right)
 &\ll
 \left(
 \frac{D_3s}
      {C_1\log k}
 \right)^s\\
 &\leq
 \left(\frac{D_3A}{C_1}\right)^s\\
 &\leq e^{-s}\\&\leq (C_2k^2)^{-1}.
\end{align*}
This proves \eqref{eqn:medium_largedeviations}--\eqref{eqn:power_largedeviations_medium}. For \eqref{eqn:power_largedeviations_tiny}, put
\[
 \mathcal Q:=\{p\leq w:p^4\mid k\},\qquad
 r:=|\mathcal Q|,\qquad m:=\lceil44\log k\rceil.
\]
If $r=0$, the probability in \eqref{eqn:power_largedeviations_tiny} is zero. Suppose $r\geq1$.
Since $\prod_{p\in\mathcal Q}p^4\mid k$, we have
\[
 r\leq\alpha\log k,\qquad \alpha:=\frac1{4\log2},
 \qquad 44\log k\leq m\leq46\log k.
\]
If $\sum_{p\in\mathcal Q}(\nu_p(\mathbf n+k)-\nu_p(k))\geq m$,
there are nonnegative integers $(h_p)_{p\in\mathcal Q}$ with
sum $m$ such that
\[
 \nu_p(\mathbf n+k)-\nu_p(k)\geq h_p
 \quad\hbox{whenever }h_p>0.
\]
There are
\[
 N=\binom{m+r-1}{r-1}\leq\binom{m+r}{r}
 \leq\left(\frac{e(m+r)}r\right)^r
 \leq\left(\frac{e(46+\alpha)}\alpha\right)^{\alpha\log k}
 \leq k^3
\]
such choices. The second last inequality follows because
$u\mapsto u\log(e(46\log k+u)/u)$ is increasing for $u>0$, and
the last inequality follows from
$\alpha\log(e(46+\alpha)/\alpha)<3$.
Applying \eqref{eqn:mainpropE} to the support of $(h_p)$ and taking a union bound,
we obtain
\begin{align*}
 \mathbb P\left(\sum_{p\in\mathcal Q}
             (\nu_p(\mathbf n+k)-\nu_p(k))\geq m\right)
 &\leq N(2^{-m}+x^{-0.3})\\
 &\leq k^{3-44\log2}+k^{-27}\\
 &\leq 3k^{-27}\\&\leq\frac{3}{2\pi^2k^2}.
\end{align*}
Here we used $k\leq x^{1/100}$ and $44\log2>30$.
Since $C_1\log k\geq m$, this proves \eqref{eqn:power_largedeviations_tiny}.
All implied constants are absolute, and $A$ depends only on $C_2$.
\end{proof}
We may now state the main properties of the random variable $\mathbf{n}$.
\begin{prop} \label{prop:mainpropaxioms}
  Let $A>1$ be fixed and let $x$ be sufficiently large in terms of $A$.
For $2\leq k\leq x^{1/100}$, define the parameters
\[
w := 0.15 \log x, \enspace 
T_k := x^{1/{10A \log k}}, \enspace K := \lfloor (\log x)^{1/1000}\rfloor, \enspace K_+ := x^{1/100},
\]
and for $1 \le k \le x^{1/100}$,
\[
R_k :=
\begin{cases}
x^{1/(100(A+1) k^{50})}, & \text{if } 1 \le k \le K,\\[6pt]
w, & \text{if } K < k \le K_+.
\end{cases}
\]
Then we may construct a random variable $\mathbf n$ taking
values in $[x,2x] \cap \mathbb{Z}$ for which the following properties hold
simultaneously for all the indicated parameters.
  \begin{enumerate}
    \item[\textnormal{(A)}] With probability 1, $\mathbf{n}$ is divisible by $p^4$ for any prime $p \le w$.
    \item[\textnormal{(B)}] If $2\leq k\leq K_+$, $1\leq j\leq s\leq2A\log k$
are integers, and $R_k<p_1,\ldots,p_j\leq T_k$ are distinct primes,
then
\[
 \mathbb P(p_1\cdots p_j\mid\mathbf n+k)
 \ll\frac{8^s}{p_1\cdots p_j}.
\]

\item[\textnormal{(C)}] If $2\leq k\leq K$, $s\geq3$ and $1\leq j\leq s$
are integers with $s\leq A\log k$, then, for some absolute
constant $C_3\geq3$,
\[
 \sum_{\substack{w<p_1,\ldots,p_j\leq R_k\\
                   p_i\ \mathrm{mutually\ distinct}}}
 \mathbb P(p_1\cdots p_j\mid\mathbf n+k)
 \ll(C_3\log s)^s.
\]

\item[\textnormal{(D)}] If $2\leq k\leq K_+$, $1\leq j\leq A\log k$ is
an integer, $a_1,\ldots,a_j\geq2$ are integers, and
$w<p_1,\ldots,p_j\leq T_k$ are distinct primes, then
\[
 \mathbb P(p_1^{a_1}\cdots p_j^{a_j}\mid\mathbf n+k)
 =\frac{\mathbb P(p_1\cdots p_j\mid\mathbf n+k)}
        {p_1^{a_1-1}\cdots p_j^{a_j-1}}
  +O(x^{-0.3}).
\]
    \item[\textnormal{(E)}] If $1 \le k \le K_+$, $\mathcal{P} \subseteq \{p \le w:p^4 \mid k\}$, and $h_p \in \mathbb{Z}^+$ for $p \in \mathcal{P}$, then
    \[
    \mathbb{P}(\nu_p(\mathbf{n}+k)-\nu_p(k) \geq h_p \ \forall p \in \mathcal{P})\le \prod_{p \in \mathcal{P}} p^{-h_p}+x^{-0.3}.
    \]
  \end{enumerate}
Here all implied constants are absolute and independent of $A$.
\end{prop}
\begin{proof}[Proof of Lemma~\ref{lem:moment_medium_large} using Proposition~\ref{prop:mainpropaxioms}]
  We first prove (\ref{eqn:moment_medium_primes}). Since there are no primes satisfying $w<p \le R_k$ if $K<k \le K_+$, it suffices to assume $k \le K$. Expanding, we get
  \[
  \mathbb{E} \left| \sum_{w<p \le R_k} \mathds{1}_{p \mid \mathbf{n}+k} \right|^{s_1}=\mathbb{E} \left[ \sum_{w<p_1,\ldots,p_{s_1} \le R_k} \mathds{1}_{p_1 \mid \mathbf{n}+k} \cdots \mathds{1}_{p_s \mid \mathbf{n}+k} \right].  
  \]
  For each $1 \le j \le s_1$, the number of terms above with exactly $j$ distinct primes amongst $p_1,\ldots,p_{s_1}$ is $\{{s_1 \atop j}\}$ times the number of terms in
  \[
  \sum_{\substack{w<p_1,\ldots,p_j \le R_k\\ p_i \text{ mutually distinct}}} \mathds{1}_{p_1 \mid \mathbf{n}+k} \cdots \mathds{1}_{p_j \mid \mathbf{n}+k}.
  \]
  Therefore, by Lemma~\ref{lem:stirling_second_kind} and Proposition \ref{prop:mainpropaxioms}(C), we have
  \begin{align*}
    \mathbb{E} \left| \sum_{w<p \le R_k} \mathds{1}_{p \mid \mathbf{n}+k} \right|^{s_1} &\ll \left( \frac{s_1}{\log s_1} \right)^{s_1}  \sum_{j=1}^{s_1}\sum_{\substack{w<p_1,\ldots,p_{j} \le R_k\\ p_i \text{ mutually distinct}}} \mathbb{P}(p_1 \cdots p_{j} \mid \mathbf{n}+k)\\ 
    &\ll (D_1s_1)^{s_1},
  \end{align*}
  for an absolute constant $D_1=2C_3$.
  We now prove (\ref{eqn:moment_large_primes}). Expanding and grouping terms, by Proposition \ref{prop:mainpropaxioms}\textnormal{(B)}, we have
\begin{align*}
    \mathbb{E} \left|\sum_{R_k<p \le T_k} \mathds{1}_{p \mid \mathbf{n}+k} \right|^{2s_2} \ll 8^{2s_2} \sum_{j=1}^{2s_2} \left\{2s_2 \atop j \right\} \left(\sum_{R_k<p \le T_k} \frac1p \right)^j.
\end{align*}
Let $M=\sum_{R_k<p \le T_k} 1/p$. If $M=0$, then \eqref{eqn:moment_large_primes} is immediate, so we assume $M>0$. Using Lemma \ref{lem:stirling_second_kind}, we have for any $z \ge 0$,
\begin{align*}
    \sum_{q \ge 0} \frac{z^q}{q!}\sum_{j=0}^{q} \left\{q \atop j \right\} M^j &= \exp \left( M\left( e^z - 1 \right) \right),
\end{align*}
so choosing $z=2s_2/(2s_2+M) \in (0,1)$, the inequality $e^z\leq(1-z)^{-1}$ gives
\[
\sum_{j=0}^{2s_2} \left\{2s_2 \atop j \right\} M^j\leq (2s_2)!z^{-2s_2}\exp(M(e^z-1))
 \leq (2s_2)!\left(\frac{M+2s_2}{2s_2}\right)^{2s_2}e^{2s_2}
 \leq\bigl(e(M+2s_2)\bigr)^{2s_2}.
\]
Thus
\[
\mathbb{E} \left|\sum_{R_k<p \le T_k}\mathds{1}_{p \mid \mathbf{n}+k}\right|^{2s_2} \ll [8e(M+2s_2)]^{2s_2}
\]
By the triangle inequality in $L^{2s_2}$,
\[
 \mathbb{E} \left|\sum_{R_k<p \le T_k} \left(\mathds{1}_{p \mid \mathbf{n}+k}-\frac1p \right) \right|^{2s_2}
 \ll \bigl(8e(M+2s_2)+M\bigr)^{2s_2}.
\]
Since $s_2\leq A\log k$, $M\ll \log k$,
\[
 8e(M+2s_2)+M
 \ll A\log k,
\]
which proves \eqref{eqn:moment_large_primes} with $D_2$ as the implied constant.
To prove \eqref{eqn:moment_power_primes_medium}, we may assume $T_k>w$, since otherwise the sum is
empty. We first bound the contribution of error terms arising from Proposition \ref{prop:mainpropaxioms}(D).
Since $\mathbf{n}+k\leq 3x$ with probability~$1$, only exponents $j$ satisfying
$$
    j\leq \frac{\log(3x)}{\log w}
$$
can contribute. Thus, on expanding the $s_3$-th moment, there are
at most
$$
    \left(T_k\frac{\log(3x)}{\log w}\right)^{s_3}
$$
terms. Grouping equal primes and taking the largest exponent
associated with each prime, we may apply Proposition \ref{prop:mainpropaxioms}(D) to
each term, since there are at most $s_3\leq A\log k$ distinct
primes. As
$$
    \frac{\log(3x)}{\log w}\leq w<T_k,
    \qquad T_k^{s_3}\leq x^{1/10}
$$
for $x$ sufficiently large, the total contribution from the error terms is
$$
    \ll x^{-0.3}
       \left(T_k\frac{\log(3x)}{\log w}\right)^{s_3}
    \leq x^{-0.3}T_k^{2s_3}
    \leq x^{-0.1}.
$$
We now upper bound the main terms. We first consider the range $R_k<p\leq T_k$. Expanding and grouping terms according to their distinct primes, a block $S_i$ with largest
exponent $j_i$ has at most $j_i^{|S_i|}$ choices of exponents.
Hence, by Proposition \ref{prop:mainpropaxioms}(D), followed by Proposition \ref{prop:mainpropaxioms}(B)
and dropping the distinctness condition, we obtain
$$
\begin{aligned}
    \mathbb{E}\left|
        \sum_{j\geq 2}\sum_{R_k<p\leq T_k}
            \mathbf{1}_{p^j\mid\mathbf{n}+k}
    \right|^{s_3}
    &\ll
    8^{s_3}
    \sum_{r=1}^{s_3}
    \sum_{\substack{
        S_1\sqcup\cdots\sqcup S_r=[s_3]\\
        S_i\neq\emptyset\\
        \min S_1<\cdots<\min S_r
    }}
    \prod_{i=1}^r
    \left(
        \sum_{j\geq 2}j^{|S_i|}
        \sum_{R_k<p\leq T_k}p^{-j}
    \right)
    +x^{-0.1}.
\end{aligned}
$$
By the exponential formula for set partitions
\[
\sum_{s\ge0}\frac{1}{s!}
\Bigg(
\sum_{\substack{r\ge 0\\ S_1\sqcup\cdots\sqcup S_r=[s]\\ S_i\neq\emptyset\\ \min S_1<\cdots<\min S_r}}
\prod_{i=1}^r W_{|S_i|}
\Bigg) z^s
=
\exp\!\left(\sum_{m\ge1}\frac{W_m}{m!}z^m\right), \quad W_m,z \geq 0
\]
and by positivity
of the coefficients, for $0<z<\log 2$ we have
$$
    \sum_{r=1}^{s_3}\sum_{\substack{
        S_1\sqcup\cdots\sqcup S_r=[s_3]\\
        S_i\neq\emptyset\\
        \min S_1<\cdots<\min S_r
    }}
    \prod_{i=1}^r
    \left(
        \sum_{j\geq 2}j^{|S_i|}
        \sum_{R_k<p\leq T_k}p^{-j}
    \right) \le s_3!z^{-s_3}
    \exp\left(
        \sum_{j\geq 2}(e^{jz}-1)
        \sum_{R_k<p\leq T_k}p^{-j}
    \right).
$$
Since $R_k\geq w\geq 4$ for $x$ sufficiently large,
$$
    \sum_{R_k<p\leq T_k}p^{-j}
    \leq \int_{R_k-1}^{\infty}t^{-j}\,dt
    =\frac{(R_k-1)^{1-j}}{j-1}
    \leq 2^{1-j}.
$$
Taking $z=1/4$, we have
$$
    \sum_{j\geq 2}(e^{j/4}-1)
        \sum_{R_k<p\leq T_k}p^{-j}
    \leq \sum_{j\geq 2}2^{1-j}(e^{j/4}-1)
    \ll 1.
$$
Using $s_3!\leq s_3^{s_3}$, we conclude that
$$
    \mathbb{E}\left|
        \sum_{j\geq 2}\sum_{R_k<p\leq T_k}
            \mathbf{1}_{p^j\mid\mathbf{n}+k}
    \right|^{s_3}
    \ll (32s_3)^{s_3}.
$$
It remains to consider $w<p\leq\min\{R_k,T_k\}$. This range
is empty if $k>K$, so we may assume $k\leq K$. Using $p_i>w$
and Proposition \ref{prop:mainpropaxioms}(C) in place of Proposition \ref{prop:mainpropaxioms}(B), the
same expansion and exponential formula give, for
$0<z<\log w$,
$$
\begin{aligned}
    \mathbb{E}\left|
        \sum_{j\geq 2}\sum_{w<p\leq \min\{R_k,T_k\}}
            \mathbf{1}_{p^j\mid\mathbf{n}+k}
    \right|^{s_3}
    &\ll
    (C_3\log s_3)^{s_3}s_3!z^{-s_3}
    \exp\left(
        \sum_{j\geq 2}\frac{e^{jz}-1}{w^{j-1}}
    \right)
    +x^{-0.1}.
\end{aligned}
$$
Taking $z=\tfrac12\log w$, we have
$$
    \sum_{j\geq 2}\frac{e^{jz}-1}{w^{j-1}}
    \leq \sum_{j\geq 2}w^{1-j/2}
    \ll 1.
$$
Therefore,
$$
\begin{aligned}
    \mathbb{E}\left|
        \sum_{j\geq 2}\sum_{w<p\leq \min\{R_k,T_k\}}
            \mathbf{1}_{p^j\mid\mathbf{n}+k}
    \right|^{s_3}
    &\ll
    \left(\frac{2C_3s_3\log s_3}{\log w}\right)^{s_3}
    +x^{-0.1}\\
    &\ll (4C_3s_3)^{s_3},
\end{aligned}
$$
where we used
$$
    s_3\leq A\log k\leq \frac{A}{100}\log x,
    \qquad \log s_3\leq 2\log w
$$
for $x$ sufficiently large in terms of $A$. Combining the two ranges proves \eqref{eqn:moment_power_primes_medium} for a sufficiently large absolute constant $D_3$. All implied constants are absolute and independent of $A$. We are done by noting that \eqref{eqn:mainpropE} is just Proposition \ref{prop:mainpropaxioms}(E).
\end{proof}
\noindent Thus, in the next section we focus on proving Proposition \ref{prop:mainpropaxioms}.

\section{Proof of Proposition \ref{prop:mainpropaxioms}}
In this section, we prove Proposition \ref{prop:mainpropaxioms}. The arguments are similar to those of \cite[Sections 2.2--2.7]{tao_quantitative_2025}, but we present the whole proof for the sake of clarity.
\subsection{Setup and Proof of Proposition~\ref{prop:mainpropaxioms}\textnormal{(A)}} \label{sec:mainpropsetup}
To construct the desired probability measure, we will proceed similarly to \cite[Section 2.2]{tao_quantitative_2025}, where a variant of a sieve of \cite{maynard_small_2015} was used. Let
\[
W := \prod_{p \le w} p^4 \le x^{0.6+o(1)} \le x^{0.61}, 
\]
for sufficiently large $x$, and let $\eta:\mathbb{R}\to[0,1]$ be a smooth function supported on $[-1,1]$ with $\eta(0)=1$. We assume $\eta$ belongs to the Gevrey class of order $2$; specifically, there exists a constant $B>1$ such that for all integers $m\ge 0$ and all $u\in\mathbb{R}$,
\begin{equation}\label{eq:gevrey_bound}
  |\eta^{(m)}(u)| \le B^m (m!)^2.
\end{equation}
Since $\eta$ is compactly supported and smooth, its Fourier transform 
\[
\widehat{\eta}(t) := \frac{1}{2 \pi} \int_\mathbb{R} \eta(u) e^{itu} \,\mathrm{d} u
\]
is well-defined. By the Fourier inversion formula, we have the representation
\[
  \eta(u)=\int_{\mathbb{R}} \widehat{\eta}(t)e^{-itu}\,\mathrm{d}t.
\]
To obtain the decay of $\widehat{\eta}(t)$, we perform $m$-fold integration by parts. For any $m \ge 1$ and $t \neq 0$, we have
\[
|\widehat{\eta}(t)| = \left| \frac{1}{2\pi (it)^m} \int_{-1}^1 \eta^{(m)}(u) e^{itu} \,\mathrm{d}u \right| \le \frac{1}{\pi |t|^m} \max_{u} |\eta^{(m)}(u)|.
\]
Applying the bound \eqref{eq:gevrey_bound} and $m! \le m^m$, it follows that
\[
|\widehat{\eta}(t)| \le \frac{1}{\pi} \left( \frac{B^m (m!)^2}{|t|^m} \right) \le \frac{1}{\pi} \left( \frac{B m^2}{|t|} \right)^m.
\]
Choosing $m = \lfloor \sqrt{|t|/4B} \rfloor$, we obtain the half-exponential decay
\[
  |\widehat{\eta}(t)| \ll \exp\bigl(-c|t|^{1/2}\bigr)
\]
for $|t| \gg 1$ and some $0<c<1$. We choose $\eta$ to be even, and hence $\widehat{\eta}$ is real and even. In addition, we also require $\widehat{\eta} \ge 0$. We can do this by first choosing any non-zero function $\eta_0:\mathbb{R} \to [0,1]$ belonging to the Gevrey class of order 2 supported on $[-1/2,1/2]$, then defining $\eta$ to be
\begin{align*}
    \eta(u) := \frac{(\eta_0*\widetilde{\eta}_0)(u)}{(\eta_0*\widetilde{\eta}_0)(0)}, \quad \text{where }\widetilde{\eta}_0(u) := \eta_0(-u).
\end{align*}
Then, $\widehat{\eta}(t) \propto |\widehat\eta_0(t)|^2 \ge 0$, as required.
Define the modified function $\widetilde{\eta}(u) := e^{-u}\eta(u)$. Then $\widetilde{\eta}$ is also supported on $[-1,1]$, satisfies $\widetilde{\eta}(0)=1$, and admits the representation
\[
  \widetilde{\eta}(u)=\int_{\mathbb{R}} \widehat{\eta}(t)e^{-(1+it)u}\,\mathrm{d}t.
\]
We will use the sieve weight
\[
\nu(n)=\mathds{1}_{n \in [x,2x]} \mathds{1}_{W \mid n} \prod_{k=1}^K \Bigg( \sum_{\substack{(d,P(w))=1\\d \mid n+k}} \mu(d) \widetilde{\eta} \left(\frac{\log d}{\log R_k}\right) \Bigg)^2  
\]
which, informally, suppresses contributions from small and medium prime factors in the shifts $n+k$ while simultaneously imposing the desired congruence conditions at the tiny primes. Clearly $\nu$ is non-negative, and it will be shown later that it is not identically zero for $x$ sufficiently large. We then define $\mathbf{n}$ by drawing $n$ from $[x,2x] \cap \mathbb{Z}$ with probability density
\[
\frac{\nu(n)}{\sum_{n'} \nu(n')}.
\]
In particular, Proposition~\ref{prop:mainpropaxioms}(A) holds.
\subsection{Initial Steps}
In this section, we prove the following lemma which sets up subsequent probability calculations.
\begin{lem}\label{lem:sieve-integral}
Let $A\geq1$ be fixed, and let $x\in\mathbb R^+$ be sufficiently
large in terms of $A$. Define
\[
K:=\lfloor(\log x)^{1/1000}\rfloor,\qquad
K_+:=x^{1/100},\qquad
w:=0.15\log x,\qquad
W:=\prod_{p\leq w}p^4,
\]
and
\[
R_k:=
\begin{cases}
x^{1/(100(A+1)k^{50})},&1\leq k\leq K,\\
w,&K<k\leq K_+.
\end{cases}
\]
Let $\eta:\mathbb R\to[0,1]$ be the fixed smooth function
chosen above, supported on $[-1,1]$, with $\eta(0)=1$,
$\widehat\eta\geq0$, and
\[
|\widehat\eta(t)|\ll \exp(-c\sqrt{|t|})
\]
for some fixed $c>0$. Recall that
$\widetilde\eta(u):=e^{-u}\eta(u)$ and
$P(w):=\prod_{p\leq w}p$. Let
\[
\nu(n):=
\mathds{1}_{n\in[x,2x]}\mathds{1}_{W\mid n}
\prod_{k=1}^K
\Bigg(
\sum_{\substack{d\mid n+k\\(d,P(w))=1}}
\mu(d)\widetilde\eta\left(\frac{\log d}{\log R_k}\right)
\Bigg)^2,
\]
and, for $2\leq k_*\leq K_+$, put
\[
P_{k_*}(d_*):=\sum_n\nu(n)\mathds{1}_{d_*\mid n+k_*}.
\]
Then, with
\[
c_0:=\int_0^\infty\bigl(\widetilde\eta'(u)\bigr)^2\d u>0,
\]
we have
\begin{equation}\label{eq:sieve-normalization}
P_{k_*}(1)
=(1+o(1))
\frac{x}{W\prod_{k=1}^K\log R_k}
\left(c_0\frac{W}{\varphi(W)}\right)^K.
\end{equation}
In particular, $P_{k_*}(1)>0$ for sufficiently large $x$.
Define $\mathbf n$ by drawing $n\in[x,2x]$ with probability
$\nu(n)/P_{k_*}(1)$. Suppose that $d_*\leq x^{1/5}$ is squarefree, all its prime
factors exceed $w$, and $j:=\omega(d_*)$. For $1\leq k\leq K$,
put
\[
\mathcal P_k:=\{p\mid d_*:p\mid k_*-k\}.
\]
For each $k$ with $\mathcal P_k\neq\varnothing$, define
\[
\begin{split}
I_k(d_*):={}&\frac1{c_0}
\iint_{|t_k|,|t_k'|\leq(\log x)^{1/5}}
\left|\frac{(1+it_k)(1+it_k')}{2+i(t_k+t_k')}\right|
\widehat\eta(t_k)\widehat\eta(t_k')\\
&\qquad\times
\prod_{p\in\mathcal P_k}
\left|
\left(1-p^{-(1+it_k)/\log R_k}\right)
\left(1-p^{-(1+it_k')/\log R_k}\right)
\right|\d t_k\d t_k'.
\end{split}
\]
Then, for some fixed $c_1>0$,
\begin{equation}\label{eq:sieve-divisibility-integral}
\mathbb P(d_*\mid\mathbf n+k_*)
\ll
\frac{2^j}{d_*}
\prod_{\substack{1\leq k\leq K\\\mathcal P_k\neq\varnothing}}
I_k(d_*)
+\frac{4^j}{d_*}\exp\bigl(-c_1(\log x)^{1/10}\bigr).
\end{equation}
Here an empty product is $1$. All implied constants are
independent of $A$, $k_*$, $d_*$ and $j$.
\end{lem}

\begin{proof}
We begin by expanding the squares in $\nu(n)$. This gives
\[
\begin{split}
P_{k_*}(d_*)={}&
\sum_{\substack{d_1,\ldots,d_K,d_1',\ldots,d_K'\\
 (d_kd_k',P(w))=1\ (1\leq k\leq K)}}
\prod_{k=1}^K
\mu(d_k)\mu(d_k')
\widetilde\eta\left(\frac{\log d_k}{\log R_k}\right)
\widetilde\eta\left(\frac{\log d_k'}{\log R_k}\right)\\
&\qquad\times
\sum_{\substack{x\leq n\leq2x\\W\mid n}}
\mathds{1}_{d_*\mid n+k_*}
\prod_{k=1}^K\mathds{1}_{[d_k,d_k']\mid n+k}.
\end{split}
\]
Since all prime factors of $d_kd_k'$ exceed $w>K$, the
congruences are compatible precisely when the integers
$[d_k,d_k']$ are pairwise coprime and
\[
\gcd([d_k,d_k'],d_*)\mid k_*-k\qquad(1\leq k\leq K).
\]
For compatible tuples, the Chinese Remainder Theorem gives
\[
\sum_{\substack{x\leq n\leq2x\\W\mid n}}
\mathds{1}_{d_*\mid n+k_*}
\prod_{k=1}^K\mathds{1}_{[d_k,d_k']\mid n+k}
=\frac{x}{Wd_*}
\prod_{k=1}^K
\frac{\gcd([d_k,d_k'],d_*)}{[d_k,d_k']}+O(1).
\]
For incompatible tuples the inner sum is zero. Since
$\widetilde\eta(u)=0$ for $u>1$ and
$|\widetilde\eta(u)|\leq1$ for $u\geq0$, the total error is at most
\[
O\left(\prod_{k=1}^K R_k^2\right)=O(x^{0.03}).
\]
For $p>w$, there is at most one $1\leq k\leq K$ for which
$p\mid k_*-k$. If it exists, denote it by $k_{*,p}$.
Fourier inversion gives
\[
\widetilde\eta(u)=\int_{\mathbb R}
\widehat\eta(t)e^{-(1+it)u}\d t.
\]
Put
\[
w_k:=\frac{1+it_k}{\log R_k},\qquad
w_k':=\frac{1+it_k'}{\log R_k},\qquad
u_{k,p}:=\frac{p^{-w_k}+p^{-w_k'}-p^{-w_k-w_k'}}p.
\]
Substituting Fourier inversion into the main term and using
multiplicativity, we obtain
\[
P_{k_*}(d_*)=
\frac{x}{W}\int_{\mathbb R^{2K}}
F(\vec t,\vec t\,')
\prod_{k=1}^K\widehat\eta(t_k)\widehat\eta(t_k')\d t_k\d t_k'
+O(x^{0.03}),
\]
where
\[
F(\vec t,\vec t\,'):=\prod_{p>w}E_{k_*,d_*,p}(\vec t,\vec t\,'),
\]
and
\[
E_{k_*,d_*,p}(\vec t,\vec t\,')=
\begin{cases}
1-\displaystyle\sum_{k=1}^K u_{k,p},&p\nmid d_*,\\[4pt]
\displaystyle\frac{(1-p^{-w_{k_{*,p}}})(1-p^{-w_{k_{*,p}}'})}{p},
 &p\mid d_*\text{ and }k_{*,p}\text{ exists},\\[6pt]
1/p,&p\mid d_*\text{ and }k_{*,p}\text{ does not exist}.
\end{cases}
\]
The sums and Euler products are absolutely convergent,
since $\Re w_k,\Re w_k'>0$, which justifies these interchanges. For all $\vec{t},\vec{t} \, ' \in \mathbb{R}^K$,
\[
|F(\vec t,\vec t\,')|
\leq\frac{4^j}{d_*}
\prod_{p>w}\left(1+\frac{3K}{p^{1+1/\log x}}\right)
\leq\frac{4^j}{d_*}\exp\bigl(O(K\log_2x)\bigr).
\]
Using the decay of $\widehat\eta$, we may therefore restrict
every integration variable to absolute value at most
$(\log x)^{1/5}$. Since $W\leq x^{0.61}$ and $d_*\leq x^{1/5}$, we have
$x/(Wd_*)\geq x^{0.19}$ for sufficiently large $x$.
Hence, for every fixed $c>0$,
\[
 x^{0.03}
 =o\left(\frac{x}{Wd_*}
          \exp\bigl(-c(\log x)^{1/10}\bigr)\right).
\]
Thus the error $O(x^{0.03})$ can be absorbed into the same bound. Therefore,
for some fixed $c>0$,
\begin{equation}\label{eq:sieve-fourier-truncation}
\begin{split}
P_{k_*}(d_*)={}&\frac{x}{W}
\int_{\substack{|t_k|,|t_k'|\leq(\log x)^{1/5}\\1\leq k\leq K}}
F(\vec t,\vec t\,')
\prod_{k=1}^K\widehat\eta(t_k)\widehat\eta(t_k')\d t_k\d t_k'\\
&\qquad+O\left(
\frac{4^jx}{Wd_*}\exp\bigl(-c(\log x)^{1/10}\bigr)
\right).
\end{split}
\end{equation}
For $1\leq k\leq K$, define
\[
G_k(t_k,t_k'):=\prod_{p>w}(1-u_{k,p})=
\prod_{p>w}
\left(
1-p^{-1-w_k}-p^{-1-w_k'}
+p^{-1-w_k-w_k'}
\right).
\]
We now estimate $G_k(t_k,t_k')$, and we will compare $F(\vec{t},\vec{t} \, ')$ to $\prod_{k=1}^K G_k(t_k,t_k')$ later. Since $\Re w_k,\Re w_k'>0$, the Euler product for the
zeta function gives
\[
\frac{\zeta(1+w_k+w_k')}
     {\zeta(1+w_k)\zeta(1+w_k')}
=
\prod_p
\frac{(1-p^{-1-w_k})(1-p^{-1-w_k'})}
     {1-p^{-1-w_k-w_k'}}.
\]
We compare the factors at primes $p>w$ first. Expanding
the numerator, we have
\[
\begin{split}
&\left(
1-p^{-1-w_k}-p^{-1-w_k'}
+p^{-1-w_k-w_k'}
\right)
\left(1-p^{-1-w_k-w_k'}\right)
-(1-p^{-1-w_k})(1-p^{-1-w_k'})\\
&\quad=
-p^{-2-w_k-w_k'}
+p^{-2-2w_k-w_k'}
+p^{-2-w_k-2w_k'}
-p^{-2-2w_k-2w_k'}\\
&\quad=
-p^{-2-w_k-w_k'}
(1-p^{-w_k})(1-p^{-w_k'}).
\end{split}
\]
Consequently,
\[
\begin{split}
&\frac{
\left(
1-p^{-1-w_k}-p^{-1-w_k'}
+p^{-1-w_k-w_k'}
\right)
\left(1-p^{-1-w_k-w_k'}\right)}
{(1-p^{-1-w_k})(1-p^{-1-w_k'})}=
1-
\frac{
p^{-2-w_k-w_k'}
(1-p^{-w_k})(1-p^{-w_k'})}
{(1-p^{-1-w_k})(1-p^{-1-w_k'})}.
\end{split}
\]
Using $\Re w_k,\Re w_k'>0$, we obtain $|p^{-w_k}|,|p^{-w_k'}|\leq1 $ and $|1-p^{-1-w_k}|,$ $|1-p^{-1-w_k'}|
\geq1-p^{-1}$, thus the absolute value of the term subtracted from $1$
is at most
\[
\frac{4}{p^2(1-1/p)^2}\leq\frac{16}{p^2}.
\]
Since
\[
\sum_{p>w}\frac1{p^2}
\leq\sum_{n>w}\frac1{n^2}\ll\frac1w,
\]
the product of these correction factors over $p>w$
is $1+O(1/w)$, where we used the inequality
\[
\left|\prod_\ell(1+a_\ell)-1\right|
\leq\prod_\ell(1+|a_\ell|)-1
\leq\exp\left(\sum_\ell|a_\ell|\right)-1,
\]
first for finite products and then by absolute convergence.
It follows that
\[
\begin{split}
G_k(t_k,t_k')
={}&
\left(1+O(1/w)\right)
\frac{\zeta(1+w_k+w_k')}
     {\zeta(1+w_k)\zeta(1+w_k')}
\prod_{p\leq w}
\frac{1-p^{-1-w_k-w_k'}}
     {(1-p^{-1-w_k})(1-p^{-1-w_k'})}.
\end{split}
\]
We next estimate the factors at primes $p\leq w$.
Expanding once more gives
\begin{align*}
\left(1-\frac1p\right)
\left(1-p^{-1-w_k-w_k'}\right)
-(1-p^{-1-w_k})(1-p^{-1-w_k'})&=
-\frac1p+p^{-1-w_k}+p^{-1-w_k'}
-p^{-1-w_k-w_k'}\\
&=
-\frac{(1-p^{-w_k})(1-p^{-w_k'})}{p}.
\end{align*}
Hence
\begin{align*}
    \frac{1-p^{-1-w_k-w_k'}}
     {(1-p^{-1-w_k})(1-p^{-1-w_k'})}
&=
\left(1-\frac1p\right)^{-1}
\left(
1-
\frac{(1-p^{-w_k})(1-p^{-w_k'})}
     {p(1-p^{-1-w_k})(1-p^{-1-w_k'})}
\right).
\end{align*}
Using $|1-e^{-z}|\leq|z|$ for $\Re z \ge 0$, we have
\[
|1-p^{-w_k}|\leq|w_k|\log p,
\qquad
|1-p^{-w_k'}|\leq|w_k'|\log p.
\]
Together with the denominator bounds above, this gives
\[
\left|
\frac{(1-p^{-w_k})(1-p^{-w_k'})}
     {p(1-p^{-1-w_k})(1-p^{-1-w_k'})}
\right|
\leq
\frac{4|w_kw_k'|(\log p)^2}{p}.
\]
We will use the estimate
\[
\sum_{p\leq w}\frac{(\log p)^2}{p}
\leq
(\log w)^2\sum_{n\leq w}\frac1n
\ll(\log w)^3.
\]
Since $|w_kw_k'|(\log w)^3=o(1)$ on the
truncated integration range, we have
\begin{align*}
    \prod_{p\leq w}
\frac{1-p^{-1-w_k-w_k'}}
     {(1-p^{-1-w_k})(1-p^{-1-w_k'})}
&=
\frac{W}{\varphi(W)}
\left(
1+O\left(|w_kw_k'|(\log w)^3\right)
\right).
\end{align*}
Thus
\begin{align*}
    G_k(t_k,t_k')
=
\frac{W}{\varphi(W)}
\frac{\zeta(1+w_k+w_k')}
     {\zeta(1+w_k)\zeta(1+w_k')}
\left(
1+O\left(\frac1w+|w_kw_k'|(\log w)^3\right)
\right).
\end{align*}
We will use the fact that the function $z\zeta(1+z)$ extends analytically
to a neighbourhood of $0$ and has value $1$ there.
Consequently,
\[
z\zeta(1+z)=1+O(|z|).
\]
Since $w_k$, $w_k'$ and $w_k+w_k'$ are nonzero, we have
\[
\begin{split}
\frac{\zeta(1+w_k+w_k')}
     {\zeta(1+w_k)\zeta(1+w_k')}
&=
\frac{w_kw_k'}{w_k+w_k'}
\frac{(w_k+w_k')\zeta(1+w_k+w_k')}
     {(w_k\zeta(1+w_k))(w_k'\zeta(1+w_k'))}\\
&=
\frac{w_kw_k'}{w_k+w_k'}
\left(1+O(|w_k|+|w_k'|)\right).
\end{split}
\]
Also,
\[
\frac{w_kw_k'}{w_k+w_k'}
=
\frac{(1+it_k)(1+it_k')}
     {(2+i(t_k+t_k'))\log R_k}.
\]
Combining these estimates, we obtain
\begin{align*}
    G_k(t_k,t_k')
={}&
\frac{W}{\varphi(W)}
\frac{(1+it_k)(1+it_k')}
     {(2+i(t_k+t_k'))\log R_k}
\left(
1+O\left(
|w_k|+|w_k'|
+\frac1w
+|w_kw_k'|(\log w)^3
\right)
\right).
\end{align*}
On the truncated integration range, we have
\[
|w_k|+|w_k'|
\ll (A+1)(\log x)^{-0.75},
\]
because
\[
|t_k|,|t_k'|\leq(\log x)^{1/5},
\qquad
\log R_k\geq\frac{(\log x)^{0.95}}{100(A+1)}.
\]
It follows that
\[
\begin{split}
|w_k|+|w_k'|+\frac1w
+|w_kw_k'|(\log w)^3&\ll
(A+1)(\log x)^{-0.75}
+(\log x)^{-1}
+(A+1)^2(\log x)^{-1.5}(\log_2x)^3\\
&\ll(\log x)^{-0.7}
\end{split}
\]
for sufficiently large $x$ in terms of $A$. Therefore,
\begin{equation}\label{eq:sieve-one-dimensional}
G_k(t_k,t_k')
=
\left(1+O\bigl((\log x)^{-0.7}\bigr)\right)
\frac{W}{\varphi(W)}
\frac{(1+it_k)(1+it_k')}
     {(2+i(t_k+t_k'))\log R_k},
\end{equation}
uniformly for $1\leq k\leq K$ throughout the truncated
integration range. Put
\[
N_k:=\iint_{|t_k|,|t_k'|\leq(\log x)^{1/5}}
G_k(t_k,t_k')\widehat\eta(t_k)\widehat\eta(t_k')\d t_k\d t_k'.
\]
Using Fourier inversion, we have
\begin{align*}
    \iint_{\mathbb R^2}
\frac{(1+it)(1+it')}{2+i(t+t')}
\widehat\eta(t)\widehat\eta(t')\d t\d t'&=
\int_0^\infty
\left(\int_{\mathbb R}
(1+it)e^{-(1+it)u}\widehat\eta(t)\d t\right)^2du
=c_0.
\end{align*}
This is positive since $\widetilde\eta(0)=1$ and
$\widetilde\eta(1)=0$.
In particular, by \eqref{eq:sieve-one-dimensional},
\begin{equation}\label{eq:sieve-local-integral}
N_k=
\left(c_0+O_\eta\bigl((\log x)^{-0.7}\bigr)\right)
\frac{W}{\varphi(W)\log R_k}.
\end{equation}
Note that $N_k$ is real, so $N_k>0$
for sufficiently large $x$. Moreover, there is a fixed
$B_\eta\geq1$ such that
\begin{equation}\label{eq:sieve-absolute-integral}
\iint_{|t_k|,|t_k'|\leq(\log x)^{1/5}}
|G_k(t_k,t_k')|\widehat\eta(t_k)\widehat\eta(t_k')\d t_k\d t_k'
\leq B_\eta N_k.
\end{equation}
Indeed, the analogous integral of
$|(1+it)(1+it')/(2+i(t+t'))|\widehat{\eta}(t)\widehat{\eta}(t')$ over $\mathbb R^2$ is finite,
and \eqref{eq:sieve-one-dimensional} and
\eqref{eq:sieve-local-integral} give the required comparison. We now compare $F(\vec{t},\vec{t} \, ')$ with $\prod_{k=1}^K G_k(t_k,t_k')$. For $p>w$, put
\[
v_{k,p}:=\frac{u_{k,p}}{1-u_{k,p}},\qquad
Q_p:=\frac{1-\sum_{k=1}^K u_{k,p}}
{\prod_{k=1}^K(1-u_{k,p})}.
\]
We have $|u_{k,p}|\leq3/p$ and $|v_{k,p}|\leq4/p$ for
sufficiently large $x$. Also,
\begin{equation}\label{eq:sieve-interaction-identity}
Q_p=1-
\sum_{\substack{S\subseteq\{1,\ldots,K\}\\|S|\geq2}}
(|S|-1)\prod_{k\in S}v_{k,p}.
\end{equation}
This follows by substituting $u_{k,p}=v_{k,p}/(1+v_{k,p})$
and expanding, noting that the coefficient corresponding to a nonempty
set $S$ is $1-|S|$. Suppose first that $d_*=1$. Then
\[
F(\vec t,\vec t\,')=\prod_{k=1}^K G_k(t_k,t_k')\prod_{p>w}Q_p.
\]
Expand each $Q_p$ using \eqref{eq:sieve-interaction-identity}.
For each term, an index $k$ not occurring in any chosen set
$S$ contributes $N_k$ after integration. For each
index that does occur, we use
\eqref{eq:sieve-absolute-integral}. Thus the contribution
of all nonconstant terms is at most
\begin{align} \label{eqn:bounddequals1byintegral}
    \left[
\prod_{p>w}\left(
1+\sum_{r=2}^K(r-1)\binom Kr
\left(\frac{4B_\eta}{p}\right)^r
\right)-1
\right]\prod_{k=1}^K N_k.
\end{align}
Since $K/w=o(1)$, we have
\[
\sum_{r=2}^K(r-1)\binom Kr
\left(\frac{4B_\eta}{p}\right)^r
\le \sum_{r=2}^K r \binom{K}{r} \left( \frac{4 B_\eta}{p} \right)^r \ll_\eta\frac{K^2}{p^2}.
\]
\eqref{eqn:bounddequals1byintegral} is therefore bounded above by
$O_\eta(K^2/w)\prod_kN_k$. 
Thus,
\[
\int_{\substack{|t_k|,|t_k'|\leq(\log x)^{1/5}\\1\leq k\leq K}}
F(\vec t,\vec t\,')
\prod_{k=1}^K\widehat\eta(t_k)\widehat\eta(t_k')\d t_k\d t_k'
=\left(1+O_\eta(K^2/w)\right)\prod_{k=1}^K N_k.
\]
Since $K(\log x)^{-0.7}=o(1)$ and $K^2/w=o(1)$,
\eqref{eq:sieve-local-integral} and
\eqref{eq:sieve-fourier-truncation} prove
\eqref{eq:sieve-normalization}. We now allow general $d_*$. For each $k$, the primes in
$\mathcal P_k$ contribute the factor
\[
\prod_{p\in\mathcal P_k}
(1-p^{-w_k})(1-p^{-w_k'}).
\]
The sets $\mathcal P_k$ are disjoint. Choose $y>w$ such that every prime dividing $d_*$ is at most $y$, and put
\[
\begin{split}
F_y(\vec t,\vec t\,')
:={}&\frac1{d_*}
\prod_{k=1}^K
\left(
G_k(t_k,t_k')
\prod_{p\in\mathcal P_k}
(1-p^{-w_k})(1-p^{-w_k'})
\right)
\prod_{\substack{w<p\leq y\\p\nmid d_*}}Q_p
\prod_{p\mid d_*}\prod_{k=1}^K(1+v_{k,p}).
\end{split}
\]
Thus $F_y$ is obtained from the factorization of $F$
by restricting the product of the $Q_p$ to $p\leq y$.
Since $Q_p=1+O(K^2/p^2)$ uniformly on the truncated
integration range, $F_y$ converges uniformly to $F$
there as $y\to\infty$. Let $\mathcal S=(S_p)_{w<p\leq y}$ range over families
of subsets of $\{1,\ldots,K\}$ satisfying
\[
\begin{cases}
S_p=\varnothing\text{ or }|S_p|\geq2,&p\nmid d_*,\\
S_p\subseteq\{1,\ldots,K\},&p\mid d_*.
\end{cases}
\]
Define
\[
c(\mathcal S):=
\prod_{\substack{w<p\leq y\\p\nmid d_*\\
                 S_p\neq\varnothing}}
(1-|S_p|).
\]
The identities
\[
Q_p=
1+
\sum_{\substack{S\subseteq\{1,\ldots,K\}\\|S|\geq2}}
(1-|S|)\prod_{k\in S}v_{k,p}
\]
and
\[
\prod_{k=1}^K(1+v_{k,p})
=
\sum_{S\subseteq\{1,\ldots,K\}}
\prod_{k\in S}v_{k,p}
\]
give
\[
\begin{split}
F_y(\vec t,\vec t\,')
={}&\frac1{d_*}\sum_{\mathcal S}c(\mathcal S)
\prod_{k=1}^K
\left(
G_k(t_k,t_k')
\prod_{p\in\mathcal P_k}
(1-p^{-w_k})(1-p^{-w_k'})
\prod_{\substack{w<p\leq y\\k\in S_p}}
v_{k,p}(t_k,t_k')
\right).
\end{split}
\]
For fixed $\mathcal S$, every factor inside the
product over $k$ depends only on $t_k,t_k'$. Hence,
if we put
\[
\begin{split}
J_k(\mathcal S):={}&
\iint_{|t_k|,|t_k'|\leq(\log x)^{1/5}}
G_k(t_k,t_k')
\prod_{p\in\mathcal P_k}
(1-p^{-w_k})(1-p^{-w_k'})
\prod_{\substack{w<p\leq y\\k\in S_p}}
v_{k,p}(t_k,t_k')
\widehat\eta(t_k)\widehat\eta(t_k')\d t_k\d t_k',
\end{split}
\]
then
\[
\begin{split}
&\int_{\substack{|t_k|,|t_k'|\leq(\log x)^{1/5}\\
                 1\leq k\leq K}}
F_y(\vec t,\vec t\,')
\prod_{k=1}^K
\widehat\eta(t_k)\widehat\eta(t_k')\d t_k\d t_k'=
\frac1{d_*}\sum_{\mathcal S}
c(\mathcal S)\prod_{k=1}^KJ_k(\mathcal S).
\end{split}
\]
We now bound $J_k(\mathcal S)$. If
$\mathcal P_k\neq\varnothing$, then
$|v_{k,p}|\leq4/p$ gives
\[
\begin{split}
|J_k(\mathcal S)|
\leq{}&
\prod_{\substack{w<p\leq y\\k\in S_p}}\frac4p
\iint_{|t_k|,|t_k'|\leq(\log x)^{1/5}}
|G_k(t_k,t_k')|
\prod_{p\in\mathcal P_k}
|(1-p^{-w_k})(1-p^{-w_k'})|
\widehat\eta(t_k)\widehat\eta(t_k')\d t_k\d t_k'.
\end{split}
\]
By \eqref{eq:sieve-one-dimensional}, the integral on
the right is at most
\[
\left(1+O_\eta\bigl((\log x)^{-0.7}\bigr)\right)
\frac{c_0W}{\varphi(W)\log R_k}I_k(d_*).
\]
Using \eqref{eq:sieve-local-integral}, we therefore have
\[
|J_k(\mathcal S)|
\leq
\left(1+C_\eta(\log x)^{-0.7}\right)
N_kI_k(d_*)
\prod_{\substack{w<p\leq y\\k\in S_p}}\frac4p
\]
for some sufficiently large fixed $C_\eta$. If $\mathcal P_k=\varnothing$ and $k$ belongs to none
of the sets $S_p$, then by definition we have
\[
J_k(\mathcal S)=N_k.
\]
If $\mathcal P_k=\varnothing$ and $k\in S_p$ for at
least one prime, then
\eqref{eq:sieve-absolute-integral} gives
\[
|J_k(\mathcal S)|
\leq B_\eta N_k
\prod_{\substack{w<p\leq y\\k\in S_p}}\frac4p.
\]
Since $B_\eta\geq1$, these three cases imply
\[
|J_k(\mathcal S)|
\leq
N_k
\prod_{\substack{w<p\leq y\\k\in S_p}}
\frac{4B_\eta}{p}
\begin{cases}
1,&\mathcal P_k=\varnothing,\\
\left(1+C_\eta(\log x)^{-0.7}\right)I_k(d_*),
   &\mathcal P_k\neq\varnothing.
\end{cases}
\]
Multiplying over $k$ and using
\[
\prod_{k=1}^K
\prod_{\substack{w<p\leq y\\k\in S_p}}
\frac{4B_\eta}{p}
=
\prod_{w<p\leq y}
\left(\frac{4B_\eta}{p}\right)^{|S_p|},
\]
we obtain
\[
\begin{split}
\prod_{k=1}^K|J_k(\mathcal S)|
\leq{}&
\left(1+C_\eta(\log x)^{-0.7}\right)^K
\prod_{k=1}^K N_k
\prod_{\substack{1\leq k\leq K\\
                 \mathcal P_k\neq\varnothing}}I_k(d_*)
\prod_{w<p\leq y}
\left(\frac{4B_\eta}{p}\right)^{|S_p|}.
\end{split}
\]
The factors on the first line do not depend on
$\mathcal S$. Therefore, we are left with evaluating
\[
\sum_{\mathcal S}|c(\mathcal S)|
\prod_{w<p\leq y}
\left(\frac{4B_\eta}{p}\right)^{|S_p|}.
\]
The choices of $S_p$ are independent for different
primes. Thus this sum equals
\[
\begin{split}
&\prod_{\substack{w<p\leq y\\p\nmid d_*}}
\left(
1+\sum_{r=2}^K
(r-1)\binom Kr
\left(\frac{4B_\eta}{p}\right)^r
\right) 
\prod_{p\mid d_*}
\left(
\sum_{r=0}^K\binom Kr
\left(\frac{4B_\eta}{p}\right)^r
\right)\\
&\quad=
\prod_{\substack{w<p\leq y\\p\nmid d_*}}
\left(
1+\sum_{r=2}^K
(r-1)\binom Kr
\left(\frac{4B_\eta}{p}\right)^r
\right)
\prod_{p\mid d_*}
\left(1+\frac{4B_\eta}{p}\right)^K.
\end{split}
\]
Thus letting $y\to\infty$, we obtain
\[
\begin{split}
\left|
\int_{\substack{|t_k|,|t_k'|\leq(\log x)^{1/5}\\
                 1\leq k\leq K}}
F(\vec t,\vec t\,')
\prod_{k=1}^K
\widehat\eta(t_k)\widehat\eta(t_k')\d t_k\d t_k'
\right|&\leq
\frac1{d_*}
\left(1+C_\eta(\log x)^{-0.7}\right)^K
\prod_{k=1}^K N_k
\prod_{\substack{1\leq k\leq K\\
                 \mathcal P_k\neq\varnothing}}I_k(d_*)\\
&\times
\prod_{\substack{p>w\\p\nmid d_*}}
\left(
1+\sum_{r=2}^K
(r-1)\binom Kr
\left(\frac{4B_\eta}{p}\right)^r
\right)
\prod_{p\mid d_*}
\left(1+\frac{4B_\eta}{p}\right)^K.
\end{split}
\]
We bound the last two products separately. 
Using $p>w$ and $K/w=o(1)$, we have
\[
\sum_{r=2}^K(r-1)\binom Kr
\left(\frac{4B_\eta}{p}\right)^r
\le \sum_{r=2}^K r \binom{K}{r} \left( \frac{4 B_\eta}{p} \right)^r \ll_\eta\frac{K^2}{p^2}.
\]
Therefore,
\[
\begin{split}
&\prod_{\substack{p>w\\p\nmid d_*}}
\left(
1+\sum_{r=2}^K
(r-1)\binom Kr
\left(\frac{4B_\eta}{p}\right)^r
\right)\leq
\exp\left(C_\eta K^2\sum_{p>w}\frac1{p^2}\right)
\leq\exp\left(C_\eta\frac{K^2}{w}\right),
\end{split}
\]
after enlarging $C_\eta$. Also,
\[
\begin{split}
\prod_{p\mid d_*}\left(1+\frac{4B_\eta}{p}\right)^K
&\leq
\exp\left(4B_\eta K\sum_{p\mid d_*}\frac1p\right)\leq
\exp\left(\frac{4B_\eta Kj}{w}\right),
\end{split}
\]
since $d_*$ has $j$ prime factors, all greater than $w$.
Combining the estimates, we obtain
\[
\begin{split}
&\left|
\int_{\substack{|t_k|,|t_k'|\leq(\log x)^{1/5}\\
                 1\leq k\leq K}}
F(\vec t,\vec t\,')
\prod_{k=1}^K
\widehat\eta(t_k)\widehat\eta(t_k')\d t_k\d t_k'
\right|\\
&\quad\leq
\frac1{d_*}
\exp\left(
C_\eta K(\log x)^{-0.7}
+C_\eta\frac{K^2}{w}
+\frac{4B_\eta Kj}{w}
\right)
\prod_{k=1}^K N_k
\prod_{\substack{1\leq k\leq K\\
                 \mathcal P_k\neq\varnothing}}I_k(d_*).
\end{split}
\]
For sufficiently large $x$,
\[
\exp\left(
C_\eta K(\log x)^{-0.7}
+C_\eta\frac{K^2}{w}
\right)\leq2,
\qquad
\frac{4B_\eta K}{w}\leq\log2.
\]
Hence
\[
\begin{split}
&\left|
\int_{\substack{|t_k|,|t_k'|\leq(\log x)^{1/5}\\
                 1\leq k\leq K}}
F(\vec t,\vec t\,')
\prod_{k=1}^K
\widehat\eta(t_k)\widehat\eta(t_k')\d t_k\d t_k'
\right|\ll_\eta
\frac{2^j}{d_*}\prod_{k=1}^K N_k
\prod_{\substack{1\leq k\leq K\\
                 \mathcal P_k\neq\varnothing}}I_k(d_*).
\end{split}
\]
It remains to divide by $P_{k_*}(1)$ and estimate the
remainder in \eqref{eq:sieve-fourier-truncation}.
The estimates used to prove \eqref{eq:sieve-normalization} give
\[
P_{k_*}(1)
=(1+o(1))\frac{x}{W}\prod_{k=1}^K N_k.
\]
Consequently, the main term after dividing by
$P_{k_*}(1)$ is
\[
\ll_\eta
\frac{2^j}{d_*}
\prod_{\substack{1\leq k\leq K\\
                 \mathcal P_k\neq\varnothing}}I_k(d_*).
\]
For the error term, \eqref{eq:sieve-local-integral}
implies for sufficiently large $x$
\[
N_k\geq
\frac{c_0}{2}\frac{W}{\varphi(W)\log R_k}
\geq\frac{c_0}{2\log x},
\]
where we used $W/\varphi(W)\geq1$ and
$\log R_k\leq\log x$. Thus
\[
\left(\prod_{k=1}^K N_k\right)^{-1}
\leq\left(\frac{2\log x}{c_0}\right)^K
\leq\exp\bigl(C_\eta K\log_2x\bigr).
\]
The error term in \eqref{eq:sieve-fourier-truncation} after dividing by $P_{k_*}(1)$ is therefore
\[
\ll_\eta
\frac{4^j}{d_*}
\exp\left(
-c(\log x)^{1/10}
+C_\eta K\log_2x
\right).
\]
Since 
$K\log_2x=o\bigl((\log x)^{1/10}\bigr)$
for sufficiently large $x$, this is at most
\[
\ll_\eta
\frac{4^j}{d_*}
\exp\left(-\frac c2(\log x)^{1/10}\right).
\]
Taking $c_1=c/2$ and combining the main term and
the error term proves
\eqref{eq:sieve-divisibility-integral}.
\end{proof}

\subsection{Proof of Proposition \ref{prop:mainpropaxioms}(B)}
In this section, we establish Proposition~\ref{prop:mainpropaxioms}\textnormal{(B)}.
\begin{lem}\label{lem:sieve-large-primes}
Let $A\geq1$, and use the parameters and random variable
of Lemma~\ref{lem:sieve-integral}. Suppose that
$2\leq k_*\leq K_+$, and put
\[
T_{k_\ast}:=x^{1/(10A\log k_*)}.
\]
Let $1\leq j\leq s\leq 2A\log k_*$ be integers, and let
$R_{k_*}<p_1<\cdots<p_j\leq T_k$ be primes. Then
\[
\mathbb P(p_1\cdots p_j\mid\mathbf n+k_*)
\ll\frac{8^s}{p_1\cdots p_j}.
\]
The implied constant is absolute and independent of $A$.
\end{lem}

\begin{proof}
Let $d_*:=p_1\cdots p_j$. Since
\[
 d_*\leq T_{k_\ast}^j\leq T_{k_\ast}^{2A\log k_*}=x^{1/5},
\]
and $R_{k_*}\geq w$ for sufficiently large $x$, we may apply Lemma~\ref{lem:sieve-integral}. First suppose that $k_*\leq K$. Since $p>w>K$, every
prime dividing $d_*$ has $k_{*,p}=k_*$. Thus
$\mathcal P_{k_*}$ consists of all $j$ primes, and all
other $\mathcal P_k$ are empty. Using
\[
|(1-p^{-w_{k_*}})(1-p^{-w_{k_*}'})|\leq4,
\]
and the finiteness of
\[
\iint_{\mathbb R^2}
\left|\frac{(1+it)(1+it')}{2+i(t+t')}\right|
\widehat\eta(t)\widehat\eta(t')\d t\d t',
\]
we obtain $I_{k_*}(d_*)\ll_\eta4^j$. Hence
\eqref{eq:sieve-divisibility-integral} gives
\[
\mathbb P(d_*\mid\mathbf n+k_*)\ll_\eta\frac{8^j}{d_*}.
\]
Now suppose that $k_*>K$. If $p\in\mathcal P_k$, then
$p\mid k_*-k$, so $p<k_*$. Since also $p\leq T_{k_\ast}$,
\[
(\log p)^2\leq(\log k_*)(\log T_{k_\ast})=\frac{\log x}{10A}.
\]
Using $|1-e^{-z}|\leq|z|$ for $\Re z\geq0$, on the
truncated integration range we have
\[
|(1-p^{-w_k})(1-p^{-w_k'})|
\leq
\frac{(1+|t_k|)(1+|t_k'|)(\log p)^2}{(\log R_k)^2}
\ll_A(\log x)^{-1/2},
\]
so $|(1-p^{-w_k})(1-p^{-w_k'})| \leq (\log x)^{-1/4}$ for $x$ sufficiently large in terms of $A$. For each nonempty
$\mathcal P_k$, we have
\[
I_k(d_*)\leq B_\eta(\log x)^{-|\mathcal P_k|/4}
\]
for some fixed $B_\eta\geq1$. For sufficiently large $x$
in terms of $A$, we have $B_\eta \le \sqrt[4]{\log x}$, and thus
\[
B_\eta(\log x)^{-|\mathcal P_k|/4}
\leq(B_\eta/\sqrt[4]{\log x})^{|\mathcal P_k|}\leq1.
\]
Therefore, \eqref{eq:sieve-divisibility-integral} gives
\[
\mathbb P(d_*\mid\mathbf n+k_*)
\ll\frac{2^j}{d_*}
+\frac{4^j}{d_*}\exp\bigl(-c_1(\log x)^{1/10}\bigr)
\ll\frac{8^j}{d_*}.
\]
Since $j\leq s$, both cases imply Proposition \ref{prop:mainpropaxioms}(B).
\end{proof}

\subsection{Proof of Proposition \ref{prop:mainpropaxioms}(C)}
In this section, we establish Proposition~\ref{prop:mainpropaxioms}\textnormal{(C)}.
\begin{lem}\label{lem:sieve-small-primes}
Let $A\geq1$, and use the parameters and random variable
of Lemma~\ref{lem:sieve-integral}. Suppose that
$2\leq k_*\leq K$, and let $s\geq3$ and $1\leq j\leq s$
be integers satisfying $s\leq A\log k_*$. Then, for some
absolute constant $C_3\geq3$ independent of $A$,
\[
\sum_{\substack{w<p_1,\ldots,p_j\leq R_{k_*}\\
                 p_i\ \mathrm{mutually\ distinct}}}
\mathbb P(p_1\cdots p_j\mid\mathbf n+k_*)
\ll(C_3\log s)^s.
\]
The implied constant is absolute and independent of $A$.
\end{lem}

\begin{proof}
For a tuple in the sum, put $d_*:=p_1\cdots p_j$. Then
\[
d_*\leq R_{k_*}^j
\leq x^{A\log k_*/(100(A+1)k_*^{50})}
\leq x^{1/100}.
\]
Thus Lemma~\ref{lem:sieve-integral} applies to every tuple.
Since $k_*\leq K$ and $p>w>K$, every prime dividing $d_*$
has $k_{*,p}=k_*$. In particular, only $\mathcal P_{k_*}$
is nonempty. We first sum the error term in
\eqref{eq:sieve-divisibility-integral}. Its total contribution
is at most
\[
4^j\exp\bigl(-c_1(\log x)^{1/10}\bigr)
\left(\sum_{w<p\leq R_{k_*}}\frac1p\right)^j.
\]
Since $j\leq A\log K\ll A\log_2x$ and
$\sum_{p\leq R_{k_*}}1/p\ll\log_2x$, this is
\[
\ll\exp\left(
-c_1(\log x)^{1/10}+O_A(\log_2x\log_3x)
\right)=o(1).
\]
We now bound the main term. For $p\leq R_{k_*}$,
\[
|(1-p^{-w_{k_*}})(1-p^{-w_{k_*}'})|
\leq2\min\left\{2,
\frac{(1+|t_{k_*}|)\log p}{\log R_{k_*}}\right\}.
\]
Using this bound in $I_{k_*}(d_*)$, extending the
nonnegative integral to $\mathbb R^2$, and dropping the
distinctness condition, we obtain
\[
\begin{split}
\sum_{\substack{w<p_1,\ldots,p_j\leq R_{k_*}\\
                 p_i\ \mathrm{mutually\ distinct}}}
\frac{2^j}{p_1\cdots p_j}I_{k_*}(p_1\cdots p_j)
&\ll_\eta4^j
\iint_{\mathbb R^2}
\left|\frac{(1+it)(1+it')}{2+i(t+t')}\right|
\widehat\eta(t)\widehat\eta(t')\\
&\quad\qquad\times
\left(
\sum_{w<p\leq R_{k_*}}\frac1p
\min\left\{2,\frac{(1+|t|)\log p}{\log R_{k_*}}\right\}
\right)^j\d t\d t'.
\end{split}
\]
Splitting the sum over primes at
$\exp(2\log R_{k_*}/(1+|t|))$, and using
\[
\sum_{p\leq y}\frac{\log p}{p}\ll\log y
\]
on the lower range and Mertens' estimate on the upper
range, we have
\[
\sum_{w<p\leq R_{k_*}}\frac1p
\min\left\{2,\frac{(1+|t|)\log p}{\log R_{k_*}}\right\}
\ll1+\log(1+|t|).
\]
Also,
\[
\left|\frac{(1+it)(1+it')}{2+i(t+t')}\right|
\leq\frac{(1+|t|)(1+|t'|)}2.
\]
Using the decay of $\widehat\eta$ and integrating over
$t'$ first, the main term is consequently bounded by
\[
C_\eta^j
\int_{\mathbb R}(1+|t|)
\bigl(1+\log(1+|t|)\bigr)^j
\exp(-c\sqrt{|t|})\d t,
\]
after enlarging a fixed constant $C_\eta$. Choose a sufficiently large fixed $D_\eta$ and split this
integral at $|t|=D_\eta(j+2)^4$. For the integral over $|t| \le D_\eta(j+2)^4$, we have
\[
1+\log(1+|t|)\ll_\eta\log(j+2),
\]
and the integral of $(1+|t|)\exp(-c\sqrt{|t|})$ is finite. For the integral over $|t|>D_\eta(j+2)^4$, we may choose $D_\eta$ so that
\[
j\log\bigl(1+\log(1+|t|)\bigr)
\leq\frac c2\sqrt{|t|}.
\]
Indeed, writing $v=\sqrt{|t|}$, the inequality
$\log(1+\log(1+v^2))\ll\sqrt v$ for $v\geq1$ shows
that a sufficiently large lower bound
$v\geq\sqrt{D_\eta}(j+2)^2$ suffices. Therefore,
\[
\int_{|t|>D_\eta(j+2)^4} (1+|t|)(1+\log(1+|t|))^j \exp \left( -c\sqrt{|t|} \right) \d t \ll \int_\mathbb{R} (1+|t|)\exp \left(-c\sqrt{|t|}/2 \right) \d t<\infty.
\]
It follows that the main term is
\[
\ll_\eta\bigl(C_\eta\log(j+2)\bigr)^j.
\]
Since $1\leq j\leq s$ and $s\geq3$, this is at most
$(C_3\log s)^s$ after choosing a sufficiently large fixed
$C_3\geq3$. Combining this with the error estimate proves
the lemma.
\end{proof}

\subsection{Proof of Proposition \ref{prop:mainpropaxioms}(D)}
In this section, we establish Proposition~\ref{prop:mainpropaxioms}\textnormal{(D)}.
\begin{lem} \label{lem:mainpropD}
Let $A\geq1$, and use the parameters and random variable
of Lemma~\ref{lem:sieve-integral}. Suppose $k_\ast$ satisfies $2 \le k_\ast \le K_+$, and put
\[
T_{k_\ast}:=x^{1/(10A\log k_*)}.
\]
Let $j$ be an integer satisfying $1 \le j \le A \log k_\ast$, $a_1,\ldots,a_j \ge 2$ be integers and $p_1,\ldots,p_j$ be distinct prime numbers satisfying
\[
w < p_1,\cdots,p_j \le T_k.
\]
Recall the sieve weight
    \[
\nu(n)=\mathds{1}_{n \in [x,2x]} \mathds{1}_{W \mid n} \prod_{k=1}^K \Bigg( \sum_{\substack{(d,P(w))=1\\d \mid n+k}} \mu(d) \widetilde{\eta} \left(\frac{\log d}{\log R_k}\right) \Bigg)^2.
\]
and let $P_{k_\ast}(d_\ast) := \sum_{n} \nu(n) \mathds{1}_{d_\ast \mid n+k_\ast}$ and $P_{k_\ast}(1) := \sum_n \nu(n)$. We define the random variable $\mathbf{n}$ taking values in $[x,2x]$ by drawing $n$ from $[x,2x]$ with probability density $\nu(n)/P_{k_\ast}(1)$. Then,
\begin{align*}
    \mathbb{P}(p_1^{a_1} \cdots p_j^{a_j} \mid \mathbf{n}+k_\ast)=\frac{P_{k_\ast}(p_1^{a_1} \cdots p_j^{a_j})}{P_{k_\ast}(1)} =\frac{\mathbb{P}(p_1 \cdots p_j \mid \mathbf{n}+k_\ast)}{p_1^{a_1-1} \cdots p_j^{a_j-1}}+O(x^{-0.3}).
\end{align*}
\end{lem}
\begin{proof}
Define
\[
  f(n)
  := \mathds{1}_{p_1^{a_1}\cdots p_j^{a_j}\mid n+k_{\ast}}
  - \frac{\mathds{1}_{p_1\cdots p_j\mid n+k_{\ast}}}{p_1^{a_1-1}\cdots p_j^{a_j-1}}.
\]
Then it suffices to show
\[
  \left| \mathbb{E}\,f(\mathbf{n}) \right|\ll x^{-0.3}.
\]
Since from \eqref{eq:sieve-normalization} we have
\[
P_{k_\ast}(1) =(1+o(1)) \frac{x}{W\prod_{k=1}^K \log R_k} \left(c_0 \frac{W}{\varphi(W)}\right)^K,
\]
it suffices to prove the estimate
\[
  \left| \sum_{n} f(n)\nu(n) \right| \ll \frac{x^{0.7}}{W \prod_{k=1}^K \log R_k} \left(c_0 \frac{W}{\varphi(W)}\right)^K.
\]
Expanding $\nu(n)$ and interchanging summations, the left-hand side becomes
\begin{align} \label{eqn:axiomdsum}
  \Bigg|\sum_{\substack{d_1,\ldots,d_K \\ d_1',\ldots,d_K' \\ (d_i, P(w)) = 1 \;\forall i\\ (d_i', P(w)) = 1 \;\forall i}} \left(\prod_{k=1}^K \mu(d_k) \mu(d_k') \widetilde{\eta} \left(\frac{\log d_k}{\log R_k}\right) \widetilde{\eta} \left(\frac{\log d_k'}{\log R_k}\right) \right)\sum_{\substack{n\in[x,2x]}} f(n) \mathds{1}_{W \mid n}
  \prod_{k=1}^{K} \mathds{1}_{[d_k,d_k']\mid n+k} \Bigg|.
\end{align}
Discard any tuple for which the divisibility conditions are incompatible.
For a remaining tuple, let $0\leq b<q$ be the residue defined by
\[
 b\equiv0\pmod W,\qquad
 b\equiv-k\pmod{[d_k,d'_k]}\quad(1\leq k\leq K),
\]
where
\[
 q:=W\operatorname{lcm}([d_1,d'_1],\ldots,[d_K,d'_K]).
\]
Put
\[
 I:=\{i:p_i\mid q\},\qquad g:=\prod_{i\in I}p_i.
\]
If $b+k_*\not\equiv0\pmod{p_i}$ for some $i\in I$, then
$f(n)=0$ throughout the progression $n\equiv b\pmod q$.
We may therefore assume $g\mid b+k_*$.
Since each $d_i,d_i'$ are squarefree and $p_i>w$, for each $i \in I$ we have
$\nu_{p_i}(q)=1$. Set
\[
 D:=\prod_{i=1}^jp_i^{a_i-1},\quad
 Q:=\prod_{i\notin I}p_i,\quad
 P':=\frac{\prod_{i=1}^jp_i^{a_i}}g=DQ,
 \quad a:=q/g,\quad b':=(b+k_*)/g.
\]
Then $\gcd(a,P')=1$. For $n=b+qm$ we have
\[
 \frac{n+k_*}{g}=am+b',\qquad
 f(b+qm)=\mathds{1}_{P'\mid am+b'}-\frac1D\mathds{1}_{Q\mid am+b'}.
\]
Each divisibility condition defines a single residue class in $m$.
Therefore,
\begin{align*}
 \sum_{\substack{x\leq n\leq2x\\n\equiv b\pmod q}}f(n)
 &=\frac{x}{qP'}-\frac1D\frac{x}{qQ}+O(1+D^{-1})=O(1),
\end{align*}
since $P'=DQ$.
Taking absolute values in \eqref{eqn:axiomdsum}, we obtain
\[
 \left|\sum_n f(n)\nu(n)\right|
 \ll\prod_{k=1}^KR_k^2\leq x^{0.03}.
\]
By \eqref{eq:sieve-normalization} and $W\leq x^{0.61}$, we have $P_{k_*}(1)\geq x^{0.38}$
for sufficiently large $x$. Hence
\[
 |\mathbb E f(\mathbf n)|\ll x^{-0.35}\ll x^{-0.3},
\]
as required.
\end{proof}
\subsection{Proof of Proposition \ref{prop:mainpropaxioms}(E)}
In this section, we establish Proposition~\ref{prop:mainpropaxioms}\textnormal{(E)}.
\begin{lem}\label{lem:mainpropE}
Let $A\geq1$, and use the parameters and random variable
of Lemma~\ref{lem:sieve-integral}. Suppose $k_\ast$ satisfies $2 \le k_\ast \le K_+$. Let $\mathcal{P} \subseteq \{p \le w:p^4 \mid k_\ast\}$, and $h_p \in \mathbb{Z}^+$ for each $p \in \mathcal{P}$. Recall the sieve weight
    \[
\nu(n)=\mathds{1}_{n \in [x,2x]} \mathds{1}_{W \mid n} \prod_{k=1}^K \Bigg( \sum_{\substack{(d,P(w))=1\\d \mid n+k}} \mu(d) \widetilde{\eta} \left(\frac{\log d}{\log R_k}\right) \Bigg)^2.
\]
and let $P_{k_\ast}(d_\ast) := \sum_{n} \nu(n) \mathds{1}_{d_\ast \mid n+k_\ast}$ and $P_{k_\ast}(1) := \sum_n \nu(n)$. We define the random variable $\mathbf{n}$ taking values in $[x,2x]$ by drawing $n$ from $[x,2x]$ with probability density $\nu(n)/P_{k_\ast}(1)$. Then,
    \[
    \mathbb{P}(\nu_p(\mathbf{n}+k_\ast)-\nu_p(k_\ast) \geq h_p \ \forall p \in \mathcal{P})\le \prod_{p \in \mathcal{P}} p^{-h_p}+x^{-0.3}.
    \]
\end{lem}
\begin{proof}
For $p\in \mathcal{P}$, let $e_p:=\nu_p(k_*)$. We begin by expanding
the squares in $\nu(n)$, as in the proof of Lemma \ref{lem:sieve-integral}.
Discard any tuple for which the congruences
\[
n\equiv 0\pmod W,\qquad
n\equiv-k\pmod{[d_k,d'_k]}\quad(1\leq k\leq K)
\]
are incompatible. For a remaining tuple, put
\[
q:=W\operatorname{lcm}([d_1,d'_1],\ldots,[d_K,d'_K]).
\]
Since all prime factors of $d_kd'_k$ exceed $w$, and
$e_p\geq4$ for every $p\in P$, the additional congruences
\[
n\equiv-k_*\pmod{p^{e_p+h_p}}\quad(p\in P)
\]
are compatible with these congruences. By the Chinese
Remainder Theorem, the combined modulus is
\[
q\prod_{p\in P}p^{e_p+h_p-4}.
\]
Therefore,
\[
\begin{aligned}
\sum_{\substack{x\leq n\leq2x\\
                 W\mid n\\
                 [d_k,d'_k]\mid n+k\ (1\leq k\leq K)}}
 \mathds{1}_{\nu_p(n+k_*)\geq e_p+h_p\ \forall p\in P} &=\frac{x}{q}\prod_{p\in P}p^{-(e_p+h_p-4)}+O(1)\\
&=\left(\prod_{p\in P}p^{-(e_p+h_p-4)}\right)
 \sum_{\substack{x\leq n\leq2x\\
                 W\mid n\\
                 [d_k,d'_k]\mid n+k\ (1\leq k\leq K)}}1
 +O(1),
\end{aligned}
\]
where we used $\prod_{p\in P}p^{-(e_p+h_p-4)}\leq1$. Multiplying by the sieve coefficients and summing over
the tuples, we obtain
\[
\begin{aligned}
\mathbb P\bigl(\nu_p(\mathbf n+k_*)-\nu_p(k_*)\geq h_p
               \ \forall p\in P\bigr)P_{k_*}(1)
&=\left(\prod_{p\in P}p^{-(e_p+h_p-4)}\right)
         P_{k_*}(1)
 +O\left(\prod_{k=1}^K R_k^2\right).
\end{aligned}
\]
Indeed, $\widetilde\eta(u)=0$ for $u>1$ and
$|\widetilde\eta(u)|\leq1$ for $u\geq0$, so the total
error is $O(\prod_{k=1}^K R_k^2)=O(x^{0.03})$.
By \eqref{eq:sieve-normalization} and $W\leq x^{0.61}$, we have
$P_{k_*}(1)\geq x^{0.38}$ for sufficiently large $x$.
Hence
\[
\begin{aligned}
\mathbb P\bigl(\nu_p(\mathbf n+k_*)-\nu_p(k_*)\geq h_p
               \ \forall p\in P\bigr)
&=\prod_{p\in P}p^{-(e_p+h_p-4)}+O(x^{-0.35})\\
&\leq\prod_{p\in P}p^{-h_p}+x^{-0.3}
\end{aligned}
\]
for sufficiently large $x$, as required.
\end{proof}
Lemmas \ref{lem:sieve-integral}, \ref{lem:sieve-large-primes}, \ref{lem:sieve-small-primes}, \ref{lem:mainpropD}, and \ref{lem:mainpropE} together prove Proposition \ref{prop:mainpropaxioms}, and thus Theorem \ref{thm:mainthm} as well.
\section{Conditional Falsity of Erd\H{o}s Problem \#679} \label{sec:conditionalfalsity}
A famous conjecture of \cite[(4)]{cramer1936order} states that if $p_n$ denotes the $n$-th prime number, then
\[
p_{n+1}-p_n \ll (\log p_n)^2.
\]
This conjecture comes from modelling primes with independent Bernoulli random variables $(X_n)_{n \ge 3}$ with $\mathbb{P}(X_n=1)=1/ \log n$ and the following lemma.
\begin{lem} \label{lem:probabilitylargestgap}
    Let $f:\mathbb{R^+} \to \mathbb{R}$ be a smooth function such that $f(n) \to \infty$ as $n \to \infty$ and $f^{(j)}(x) \ll_j x^{-j}$ for all $j \ge 1$. Choose an integer $n_0$ such that $f(n)\geq 1$ for all
$n\geq n_0$. Let $(X_n)_{n\geq n_0}$ be a sequence of
independent Bernoulli random variables with
$\mathbb{P}(X_n=1)=1/f(n)$. Let $S_1< S_2< \cdots$ enumerate all indices $n$ for which $X_{n}=1$. Then, with probability 1
    \[
    \limsup_{k \to \infty} \frac{S_{k+1}-S_k}{f(S_k) \log S_k} \le 1.
    \]
\end{lem}
\begin{proof}
    This follows from the Borel-Cantelli lemma.
\end{proof}

For $k \in \mathbb{N}$, let $\pi_k(x)$ denote the number of integers not greater than $x$ with exactly $k$ distinct prime factors. We quote the following result in \cite[Chapter II.6, Theorem 6.4]{tenenbaum2015introduction}.
\begin{thm} \label{thm:densitysemiprimes}
    Let $A>0$. For $x$ sufficiently large in terms of $A$, uniformly for $1 \le k \le A \log_2 x$, we have
    \[
    \pi_k(x) \gg_A \frac{x}{ \log x} \frac{(\log_2 x)^{k-1}}{(k-1)!}.
    \]
\end{thm}
By putting $k=\lceil \log_2 x \rceil$ into Theorem \ref{thm:densitysemiprimes} and using Stirling's approximation, we see that
\[
\pi_k(x) \gg \frac{x}{\sqrt{\log_2 x}}.
\]
Therefore, treating membership of $\{n \le x:\omega(n) = \lceil \log_2 x \rceil \}$ as approximately independent with probability of order $1/\sqrt{\log_2 x}$, using Lemma \ref{lem:probabilitylargestgap} we are motivated to conjecture the following.
\begin{conj} \label{conj:strongaveragenumbers}
    Let
\[
 \mathcal A:=\{n\in\mathbb N:n\geq16,\ \omega(n)\geq\log_2 n\}.
\]
There exists a constant $C>0$ such that, for every sufficiently
large real $x$,
\[
 \mathcal A\cap
 \left(x-C\log x\sqrt{\log_2 x},\,x\right]\ne\varnothing.
\]
\end{conj}
In fact, a weaker version of Conjecture \ref{conj:strongaveragenumbers} suffices to disprove Conjecture \ref{conj:erdos679}.
\begin{conj}\label{conj:weakaveragenumbers}
There exist constants $C_0>d\geq1$ such that, on defining
\[
 \mathcal A:=\left\{m\in\mathbb N:m\geq16,\
       \omega(m)\geq C_0\frac{\log_2m}{\log_3m}\right\},
\]
we have, for every sufficiently large real $x$,
\[
 \mathcal A\cap\left(x-(\log(x/2))^d,\,x\right]\ne\varnothing.
\]
\end{conj}
\begin{thm}\label{thm:falseerdos}
Assume Conjecture \ref{conj:weakaveragenumbers}. Then there exists $\delta>0$ such that, for
every fixed integer $H\geq3$ and every sufficiently large integer
$n$ in terms of $H$, there is an integer $H\leq k<n$ satisfying
\[
 \omega(n-k)>(1+\delta)\frac{\log k}{\log\log k}.
\]
In particular, Conjecture \ref{conj:erdos679} is false.
\end{thm}
\begin{proof}
Fix constants $C_0>d\geq1$ from Conjecture \ref{conj:weakaveragenumbers}, and choose
$\delta>0$ with $1+\delta<C_0/d$. Fix an integer $H\geq3$.
Apply Conjecture \ref{conj:weakaveragenumbers} at $x=n-H$. It gives an integer $m$ such that
\[
 n-H-(\log((n-H)/2))^d<m\leq n-H,
 \qquad
 \omega(m)\geq C_0\frac{\log_2m}{\log_3m}.
\]
Set $k=n-m$ and $U:=H+(\log((n-H)/2))^d$. Then
\[
 H\leq k<U<n,
 \qquad m=n-O_H((\log n)^d).
\]
Therefore,
\[
 \frac{\log_2m}{\log_3m}
 =(1+o(1))\frac{\log_2n}{\log_3n},
 \qquad
 \frac{\log U}{\log\log U}
 =(d+o(1))\frac{\log_2n}{\log_3n}.
\]
Hence, for $H\leq k<U$,
\[
 \frac{\log k}{\log\log k}
 \leq(d+o(1))\frac{\log_2n}{\log_3n}.
\]
Since $C_0>(1+\delta)d$, we obtain, for sufficiently large $n$,
\[
 \omega(n-k)\geq C_0\frac{\log_2m}{\log_3m}
 >(1+\delta)\frac{\log k}{\log\log k}.
\]
As $H$ was arbitrary, this contradicts
Conjecture \ref{conj:erdos679}.
\end{proof}
Using an analogous argument, Conjecture \ref{conj:strongaveragenumbers} implies Conjecture \ref{conj:newerdos679}. Indeed, fix $\varepsilon>0$
and apply Conjecture \ref{conj:strongaveragenumbers} at $x=n-2$. For sufficiently large $n$,
there is an integer $m\in\mathcal A$ such that, with $k=n-m$,
\[
 2\leq k\leq2+C\log(n-2)\sqrt{\log_2(n-2)}<n.
\]
Since $m=n-O(\log n\sqrt{\log_2 n})$, we have
\[
 \omega(n-k)\geq\log_2(n-k)=\log_2n+o(1),
 \qquad
 \log k\leq\log_2n+\tfrac12\log_3n+O(1).
\]
It follows that $\omega(n-k)>(1-\varepsilon)\log k$ for
sufficiently large $n$, as required.

The same random model also suggests that the exponent $1$
in Conjecture \ref{conj:erdos826} is optimal. Indeed, taking
$k=\lceil 2\log_2 x\rceil$ in Theorem \ref{thm:densitysemiprimes} and applying
Stirling's approximation gives
\[
    \pi_k(x)\gg
    \frac{x}{(\log x)^{2\log 2-1}\sqrt{\log_2 x}}.
\]
Thus, arguing as above, we expect that for every $\delta>0$
and every sufficiently large real $x$, the interval
\[
    \bigl(x-(\log x)^{2\log 2+\delta},x\bigr]
\]
contains an integer $m$ with $\omega(m)\ge 2\log_2 m$. Assuming this, fix $0<\varepsilon<1$
and choose $\delta>0$ sufficiently small that
\[
    (1-\varepsilon)(2\log 2+2\delta)<2\log 2.
\]
Apply the above with
$x=n+(\log n)^{2\log 2+2\delta}$, and put $k:=m-n$.
For sufficiently large $n$, we have
$2\le k\le(\log n)^{2\log 2+2\delta}$, while
\[
    \tau(n+k)\ge 2^{\omega(m)}
    \ge(\log m)^{2\log 2}
    \ge(\log n)^{2\log 2}.
\]
Therefore,
\[
    \frac{\tau(n+k)}{k^{1-\varepsilon}}
    \ge
    (\log n)^{2\log 2-(1-\varepsilon)(2\log 2+2\delta)}
    \longrightarrow\infty,
\]
as $n \to \infty$. Hence this implies that the bound
in Conjecture \ref{conj:erdos826} cannot be replaced by
$\tau(n+k)\ll k^{1-\varepsilon}$ for any fixed
$\varepsilon>0$.

\setcitestyle{numbers}
\bibliographystyle{alpha}

\end{document}